\documentclass[11pt]{amsart}

\usepackage[mathscr]{eucal}
\usepackage{amsmath,amssymb,amsfonts,amsthm,enumerate}

\newtheorem{theorem}{Theorem}[section]
\newtheorem{lemma}[theorem]{Lemma}
\newtheorem{corollary}[theorem]{Corollary}

\newtheorem{theirtheorem}{Theorem}

\newcommand{\Z}{\mathbb Z}
\newcommand{\R}{\mathbb R}

\DeclareMathOperator{\ord}{ord}

\DeclareMathOperator{\supp}{Supp}

\newcommand{\la}{\langle}
\newcommand{\ra}{\rangle}
\newcommand{\be}{\begin{equation}}
\newcommand{\ee}{\end{equation}}
\newcommand{\und}{\;\mbox{ and }\;}
\newcommand{\nn}{\nonumber}
\newcommand{\ber}{\begin{eqnarray}}
\newcommand{\eer}{\end{eqnarray}}
\newcommand{\Sum}[2]{\underset{#1}{\overset{#2}{\sum}}}
\newcommand{\Summ}[1]{\underset{#1}{\sum}}

\newcommand{\barr}[2]{\overline{{#1\;}^{#2}}}

\newcommand{\A}{\mathscr A}

\newcommand{\sP}{\mathscr S}
\newcommand{\Fc}{\mathcal F}
\newcommand{\vp}{\mathsf v}
\newcommand{\h}{\mathsf h}

\newcommand{\sA}{\mathscr{A}}

\DeclareSymbolFont{goo}{OMS}{cmsy}{b}{n}
\DeclareMathSymbol{\gooT}{\mathalpha}{goo}{"1}
\newcommand{\bdot}{\mathbin{\gooT}}
\begin{document}

\title{Iterated Sumsets and  Subsequence Sums}
\author{David J. Grynkiewicz}
\email{diambri@hotmail.com}
\address{Department of Mathematical Sciences, University of Memphis, Memphis,  TN 38152, USA}

\begin{abstract}
Let $G\cong \Z/m_1\Z\times\ldots\times \Z/m_r\Z$ be a finite abelian group with $1<m_1\mid\ldots\mid m_r=\exp(G)$. The Kemperman Structure Theorem characterizes all subsets $A,\,B\subseteq G$  satisfying $|A+B|<|A|+|B|$ and has been  extended to cover the case when $|A+B|\leq |A|+|B|$. Utilizing these results, we provide a precise structural description of all finite subsets $A\subseteq G$ with  $|nA|\leq (|A|+1)n-3$ when $n\geq 3$ (also when $G$ is infinite), in which case many of the pathological possibilities from the case $n=2$ vanish, particularly for large $n\geq \exp(G)-1$. The structural description is combined with other arguments to generalize  a subsequence sum result of Olson
 asserting that a sequence $S$ of terms from $G$ having length $|S|\geq 2|G|-1$ must either have every element of $G$ representable as a sum of $|G|$-terms from $S$ or else have all but $|G/H|-2$ of its terms lying in a common $H$-coset for some $H\leq G$. We show that the much weaker hypothesis $|S|\geq |G|+\exp(G)$ suffices to obtain a nearly identical conclusion, where for the case $H$ is trivial we must allow all but $|G/H|-1$ terms of $S$ to be  from the same $H$-coset. The bound on $|S|$ is improved for several classes of groups $G$, yielding optimal lower bounds for $|S|$. We also generalize Olson's result for $|G|$-term subsums to an analogous one for $n$-term subsums when $n\geq \exp(G)$, with the bound likewise improved for several special classes of groups. This improves previous generalizations of Olson's result, with the bounds for  $n$ optimal.\end{abstract}

\maketitle

\section{Notation and Overview}

\subsection{Notation} Let $G$ be an abelian group and let $A,\,B\subseteq G$ be finite and nonempty subsets. Their sumset is defined as $A+B=\{a+b:\;a\in A,\,b\in B\}$. For $x\in G$, we let $\mathsf r_{A+B}(x)=|(x-B)\cap A|=|(x-A)\cap B)|$ denote the number of ways to represent $x=a+b$ as an element in the sumset $A+B$, where $(a,b)\in A\times B$. When $\mathsf r_{A+B}(x)=1$, we say that $x$ is a \emph{unique expression} element in $A+B$. Note $A+B=\{x\in G:\;\mathsf r_{A+B}(x)\geq 1\}$. Multiple summand sumsets are defined analogously: $\Sum{i=1}{n}A_i=\{\Sum{i=1}{n}a_i:\;a_i\in A_i\}$ for subsets $A_1,\ldots,A_n\subseteq G$. For an integer $n\geq 0$, we use the abbreviation $nA={\underbrace{A+\ldots+A}}_{n}$, where $0A:=\{0\}$, for the $n$-fold iterated sumset.

The \emph{stabilizer} of $A\subseteq G$ is the subgroup $\mathsf H(A)=\{x\in G:\; x+A=A\}\leq G$. It is the maximal subgroup $H$ such that $A$ is a union of $H$-cosets. When $\mathsf H(A)$ is trivial,  $A$ is called \emph{aperiodic}, and when $\mathsf H(A)$ is nontrivial,  $A$ is called \emph{periodic}. More generally, if $A$ is a union of $H$-cosets for some subgroup $H\leq G$ (necessarily with $H\leq \mathsf H(A)$), then $A$ is called \emph{$H$-periodic}. Observe that $A$ being $H$-periodic implies that $A+B$ is also $H$-periodic for any nonempty subset $B\subseteq G$. In particular, if $nA$ is aperiodic, then so is $kA$ for any $k\leq n$.

If $H\leq G$ is a subgroup, then we let $$\phi_H:G\rightarrow G/H$$ denote the natural homomorphism. Note, if $H=\mathsf H(A)$, then $\phi_H(A)$ is aperiodic. We use $H<G$ to indicate that $H$ is proper, and $$\la A\ra_*:=\la A-A\ra=\la -x+A\ra\quad\mbox{ for any $x\in A$}$$ denotes the subgroup generated affinely by $A$, which is the smallest subgroup $H$ such that $A$ is contained in an $H$-coset.
The relative complement of $A$ is defined as $$\barr{A}{H}:=(H+A)\setminus A.$$ When the subgroup $H$ is implicit, it will usually be dropped from the notation.

Regarding sequences and subsequence sums, we follow the standardized notation from Factorization Theory  \cite{gao-ger-survey} \cite{alfredbook} \cite{Gbook}.  The key parts are summarized  here. Let $G_0\subseteq G$ be a subset. A \emph{sequence} $S$ of terms from $G_0$ is viewed formally as an element of the free abelian monoid  with basis $G_0$, denoted  $\mathcal F(G_0)$. Thus a sequence $S\in \Fc(G_0)$ is written as a finite multiplicative string of terms, using the bold dot operation $\bdot$ to concatenate terms, and with the order irrelevant:
$$S=g_1\bdot\ldots\bdot g_\ell$$ with $g_i\in G_0$ the terms of  $S$ and $|S|:=\ell\geq 0$ the \emph{length} of $S$.
 Given $g\in G_0$ and $s\geq 0$,  we let $g^{[s]}={\underbrace{g\bdot\ldots\bdot g}}_s$ denote the sequence consisting of the element $g$ repeated $s$ times.  We let
$$\vp_g(S)=|\{i\in [1,\ell]:\;g_i=g\}|\geq 0$$
denote the multiplicity of the term $g\in G_0$ in the sequence $S$.
%Likewise, for a subset $X\subseteq G_0$, $\vp_X(S)=\Summ{g\in X}\vp_g(S)$ denotes the number of terms of $S$ from the set $X$.
If $S,\,T\in \Fc(G_0)$ are sequences, then  $S\bdot T\in \Fc(G_0)$ is  the sequence obtained by concatenating the terms of $T$ after those of $S$.
  A sequence $S$ may also be defined by listing its terms as a product: $S=\prod^\bullet_{g\in G_0}g^{[\vp_g(S)]}.$
 We use $T\mid S$ to indicate that $T$ is a subsequence of $S$ and let ${T}^{[-1]}\bdot S$ or $S\bdot {T}^{[-1]}$ denote the sequence obtained by  removing the terms of $T$ from $S$.
Then
\begin{align*}
 &\h (S) =     \max \{ \mathsf v_g(S):\; g \in G_0 \} \quad  \mbox{is the \emph{maximum  multiplicity} of $S$},\\
 &\supp(S)=\{g\in G_0:\; \vp_g(S)>0\}\subseteq G \quad\mbox{ is the \emph{support} of $S$},\\
 &\sigma(S)=\Sum{i=1}{\ell}g_i=\Summ{g\in G_0}\vp_g(S)g\in G\quad\mbox{ is the \emph{sum} of $S$},\\
 &\Sigma_n(S)=\{\sigma(T):\; T\mid S, \, |T|=n\}\subseteq G \quad \mbox{ are the $n$-term \emph{sub(sequence)-sums} of $S$}.
\end{align*}
Given a map
$\varphi \colon G_0 \to G'_0$, we let $\varphi(S)=\varphi(g_1)\bdot\ldots\bdot \varphi(g_\ell)\in \Fc(G'_0)$. The sequence $S$ is called \emph{zero-sum} if $\sigma(S)=0$. A \emph{setpartition} $\sA=A_1\bdot\ldots\bdot A_n$ over $G_0$ is a sequence of \emph{finite}, \emph{nonempty} subsets $A_i\subseteq G_0$. A setpartition naturally partitions its underlying sequence $$\mathsf{S}(\mathscr A):={\prod}^\bullet_{i\in [1,n]}{\prod}^\bullet_{g\in A_i}g\in \Fc(G_0)$$ into $n$ sets, so $\mathsf S(\sA)$ is the sequence obtained by concatenating the elements from every  $A_i$. We let $\sP(G_0)$ denote the set of all setpartitions over $G_0$, and refer to a setpartition of length $|\A|=n$ as an \emph{$n$-setpartition}.

Intervals are discrete, so $[a,b]=\{x\in \Z:\; a\leq x\leq b\}$ for $a,\,b\in \R$, as are variables introduced with inequalities. For $m\geq 1$, we let $C_m\cong \Z/m\Z$ denote a cyclic group of order $m$. If $G$ is finite, then $G\cong C_{m_1}\times\ldots\times C_{m_r}$ for some $m_1\mid \ldots\mid m_r$ with $m_r=\exp(G)$ the \emph{exponent} of $G$.  For $G$ cyclic,  an affine transformation is a map $\varphi:G\rightarrow G$ of the form $\varphi(x)=sx+y$ for $x\in G$, where $y\in G$, $s\in \Z$ and $\gcd(s,|G|)=1$. The \emph{Davenport Constant}, denoted $\mathsf D(G)$, is the least integer such  that a sequence of terms from $G$ with length $|S|\geq \mathsf D(G)$ must always contain a nontrivial zero-sum subsequence. In general, $\mathsf D^*(G)\leq \mathsf D(G)\leq |G|$, where $\mathsf D^*(G):=1+\Sum{i=1}{r}(m_i-1)$, though both inequalities may be strict (see \cite[Propositions 5.1.4 and 5.1.8, pp. 341]{alfredbook}, or \cite{wolfgang-ordaz-olson-constant} for related results regarding the strong Davenport constant).

\subsection{Overview} Inverse structure theorems for sumsets, describing the structure of the summands $A$ and $B$ when $|A+B|$ is small in comparison to the size of $|A|$ and $|B|$, are among the most fundamental questions in Additive Combinatorics. The texts \cite{Alfred-Ruzsa-book} \cite{alfredbook} \cite{Gbook} \cite{natboook} \cite{taobook} provide some overview.  While there are many such results approximating the structure of $A$ and $B$, particularly in special groups, there are very few that fully characterize the possibilities, especially for an unrestricted  abelian group $G$. One such result is due to Kemperman \cite[Chapter 9]{Gbook}  \cite{ham-kst} \cite{kemp}, who  gave a full characterization of when $|A+B|< |A|+|B|$. This was later extended to a  characterization of when $|A+B|\leq |A|+|B|$ in \cite{kst+1}, generalizing partial work achieved in \cite{ham-pre-kst+1}. They include some unwieldy possibilities, particularly when $|A+B|$ is large in comparison to $|G|$, leading us to defer the relevant details until  Section \ref{sec-CriticalPair-Theory}. Our first goal in this paper is to extend the symmetric case in these results to  $n$-fold iterated sumsets, giving the following precise characterization applicable when $|nA|<n|A+H|+(n-3)|H|,$ where  $H=\mathsf H(nA)$, by applying it to $n\phi_H(A)$. The definitions used to describe the possible structures in Theorem \ref{thm-main-sumset} are explained in detail in Section \ref{sec-CriticalPair-Theory}.

\begin{theorem}\label{thm-main-sumset}
Let $G$ be a nontrivial  abelian group,  let $A\subseteq G$ be a finite subset with $\la A\ra_*=G$, and let $n\geq 3$ be an integer. Suppose $nA$ is aperiodic and $$|nA|<(|A|+1)n-3.$$ If $|A|=3$, then $A$ is given by one of the possibilities listed in Lemma \ref{lemma-itt-kst-size3}. Otherwise, one of the following must hold.
\begin{itemize}
\item[(i)]   There is an arithmetic progression $P\subseteq G$ such that $A\subseteq P$ and $|P|\leq |A|+1$, in which case   $|nA|=(|A|-1)n+1$, $|nA|=|A|n$, $|nA|=|A|n+1$ or $|nA|=|A|n-1=|G|-1$.

\item[(ii)] There are subgroups $K_1,\,K_2,\,H< G$ with $H=K_1\oplus K_2\cong C_2\oplus C_2$ such that $$z+A=\Big(x+K_1\Big)\cup \Big (y+H\Big)\cup\ldots\cup \Big((r-1)y+H\Big)\cup \Big (ry+K_2\Big)\quad\mbox{ with } r\geq 1,$$ for some $z\in G$, $x\in H$ and $y\in G\setminus H$, in which case $|nA|=|A|n$ or $|nA|=|A|n-1=|G|-1$.

\item[(iii)] There is a subgroup $H<G$ with $|H|=2$ such that $$z+A=\{x\}\cup \Big(y+H\Big)\cup \ldots\cup \Big(ry+H\Big)\cup \{(r+1)y\}\quad\mbox{ with $r\geq 1$},$$ for some $z\in G$, $x\in H$ and $y\in G\setminus H$,  in which case $|nA|=|A|n$ or $|nA|=|A|n-1=|G|-1$.

\item[(iv)] There is a nontrivial subgroup $H<G$  such that $$z+A=\{0\}\cup \Big(y+(H\setminus \{x\})\Big)\cup \Big(2y+H\Big)\cup \ldots\cup \Big(ry+H\Big)\quad\mbox{ with $r\geq 1$},$$ for some $z\in G$, $x\in H$ and  $y\in G\setminus H$, with $r\geq 2$ when $|H|=2$, in which case $|nA|=|A|n$ or $|nA|=|A|n-1=|G|-1$.

\item[(v)] There is a nontrivial subgroup $H<G$, nonempty $A_0\subseteq H$ and set $$P=A_0\cup (y+H)\cup \ldots\cup (ry+H)\quad\mbox{ with $r\geq 1$},$$ for some $y\in G\setminus H$, such that
    \begin{itemize}
  \item[(a)]  $A_0\subseteq z+A\subseteq P$ with $|P|=|A|+\epsilon\leq |A|+1$, for some $z\in G$,
   \item[(b)] $nA_0$ is aperiodic,
   \item[(c)] either $|A_0|=1$ or $|nA_0|<\min\{|\la A_0\ra_*|,\, (|A_0|+1-\epsilon)n-3\}$,
   \item[(d)]       $nA\setminus nA_0$ is $H$-periodic, and
   \item[(e)] $|nA|-|A|n=|nA_0|-|A_0|n+\epsilon n$.
\end{itemize}
\end{itemize}
\end{theorem}

When $G$ is finite and $n$ is large, many of the possibilities given in Theorem \ref{thm-main-sumset} are no longer possible. Requiring  $|nA|<|A|n$ or knowing that $|A|$ is large can also further simplify the list of structures. We give two such corollaries (Corollaries \ref{cor1} and \ref{cor2}) at the end of Section \ref{sec-itterated-sumsets}, though many more would be possible, varying according to the specific limitations imposed on  $n$, $A$ and $G$.

Our second goal is to utilize the structural characterization given in Theorem \ref{thm-main-sumset} to help improve some classical results regarding $n$-term subsequence sums and zero-sums.
One inception for the study of subsequence sums is the Erd\H{o}s-Ginzburg-Ziv Theorem \cite{egz} \cite[Corollary 4.2.8]{Alfred-Ruzsa-book} \cite[Corollary 5.7.5]{alfredbook} \cite[Theorem 10.1]{Gbook}.

\begin{theirtheorem}[Erd\H{o}s-Ginzburg-Ziv Theorem] Let $G$ be a finite abelian group and let $S\in \Fc(G)$ be a sequence of terms from $G$ of length $|S|\geq 2|G|-1$. Then $0\in \Sigma_{|G|}(S)$.
\end{theirtheorem}

If one is interested in knowing whether an element $g\in G$ other than $0$ can be represented as a subsum, there is a natural obstruction: $S$ could consist of a single element repeated with high multiplicity or, more generally, most of the terms of $S$ could lie in a coset of a proper subgroup. Olson \cite{olson1-pregao}, generalizing previous work of Mann \cite{Mann-preolson}, showed this  to be the only  barrier. We refer to the hypothesis in Theorem \ref{thm-olson} that, for every $H<G$ and $\alpha\in G$, there are at least $|G/H|-1$ terms of $S$ lying outside the coset $\alpha+H$, as the \emph{coset condition}.

\begin{theirtheorem}\label{thm-olson}\cite{olson1-pregao}
Let $G$ be a finite abelian group and let $S\in \Fc(G)$ be a sequence of terms from $G$ of length $|S|\geq 2|G|-1$. Suppose, for every $H<G$ and $\alpha\in G$, there are at least $|G/H|-1$ terms of $S$ lying outside the coset $\alpha+H$. Then  $\Sigma_{|G|}(S)=G$.
\end{theirtheorem}

There are several natural approaches to generalizing the above result of Olson. First, one could ask whether the bound $|S|\geq 2|G|-1$ is tight. Second, one could attempt to replace $\Sigma_{|G|}(S)$ with $\Sigma_n(S)$ for a more general integer $n\geq 1$. Third, one could ask whether the bound $|G/H|-1$ is tight. Towards this end, there are results addressing the first two approaches.

For instance, Gao \cite{Gao-preolson} showed that the hypothesis $|S|\geq 2|G|-1$ in Theorem \ref{thm-olson} could be replaced by $|S|\geq |G|+\mathsf D(G)-1$, and the theorem remained true.  A basic argument \cite[Theorem 10.2]{Gbook} shows $\mathsf D(G)\leq |G|$, while  $\mathsf D(G)$ is generally much smaller than $|G|$ (see \cite[Theorem 5.5.5]{alfredbook} \cite[Theorem 17.1]{Gbook}). Thus the result of Gao provided a strong generalization of Olson's result. However, the bound $|S|\geq |G|+\mathsf D(G)-1$ is not optimal. It was later shown in \cite{oscar-weighted} that the hypothesis  $|S|\geq |G|+\mathsf D(G)-1$ could be replaced by $|S|\geq |G|+\mathsf D^*(G)-1$. Since $\mathsf D^*(G)\leq \mathsf D(G)$ is the basic lower bound for the Davenport constant, known to be strictly tight in many instances (see \cite[pp. 341]{alfredbook}), this was an improvement on the bound given by Gao. It naturally raises the question, what is the minimal integer $n_G$ such that $|S|\geq |G|+n_G$ implies $\Sigma_{|G|}(S)=G$, assuming  the coset condition given in Theorem \ref{thm-olson} holds? Note the coset condition failing for a coset of the subgroup  $K<G$ implies that $|S|\leq \mathsf h(S)|K|+|G/K|-2$, which will be useful for showing that the coset condition holds in the following examples.

\subsection*{Example A.1} If $G=\la g\ra$ is cyclic of composite order   with $p$ the smallest prime divisor of $|G|$ and $H\leq G$ the subgroup of order $p$, then the sequence $S= g^{[\frac1p|G|-1]}\bdot \prod^\bullet_{h\in H}h^{[\frac1p|G|]}$ has $|S|=|G|+\frac{1}{p}|G|-1$, satisfies the coset condition, and yet $\Sigma_{|G|}(S)=\sigma(S)-\Sigma_{|G|/p-1}(S)\neq G$. This shows that we can do no better than $n_G\geq |G|/p$ when $G$ is cyclic.

\subsection*{Example A.2} If $G$ is non-cyclic, then $G=H\oplus\la g\ra\cong H\oplus  C_{\exp(G)}$ for some nontrivial subgroup $H< G$. In this case, the sequence $S= g^{[\exp(G)-1]}\bdot \prod_{h\in H}^{\bullet}h^{[\exp(G)]}$ has $|S|=|G|+\exp(G)-1$, satisfies the coset condition, and yet $\Sigma_{|G|}(S)=\sigma(S)-\Sigma_{\exp(G)-1}(S)\neq G$. Thus
we can do no better than $n_G\geq \exp(G)$ when $G$ is non-cyclic.

\subsection*{Example A.3} If $G$ is neither cyclic nor isomorphic to $C_2^2$ and $G=H\oplus\la g\ra\cong H\oplus  C_{\exp(G)}$ with $|H|\geq \exp(G)$, then  the sequence $S'=\prod_{h\in H\setminus \{0\}\cup\{g\}}^{\bullet}h^{[\exp(G)+1]}$ has $|S'|=|G|+|H|\geq |G|+\exp(G)$ and $\Sigma_{\exp(G)}(S')\neq G$. Thus, any subsequence $S\mid S'$ with $|S|=|G|+\exp(G)$ will have $\Sigma_{|G|}(S)=\sigma(S)-\Sigma_{\exp(G)}(S)\subseteq \sigma(S)-\Sigma_{\exp(G)}(S')\neq G$. If we choose $S$ such that $\vp_h(S)\geq \exp(G)$ for all $h\in \supp(S')$, then $S$ will also satisfy the coset condition (since $G\not\cong C_2^2$), showing we can do no better than $n_G\geq \exp(G)+1$ when $|G|\geq \max\{5,\, \exp(G)^2\}$.

\medskip

In all the above examples, we have made use of the general fact that $\Sigma_n(S)=\sigma(S)-\Sigma_{|S|-n}(S)$, which follows in view of the  one-to-one correspondence between a subsequence $T\mid S$ of length $|T|=n$ and its complementary sequence $S\bdot {T}^{[-1]}$. If one is interested in studying the set of $n$-term subsums $\Sigma_n(S)$, then having a term with multiplicity greater than $n$ is no better than having the same term with multiplicity equal to $n$. In other words,
$\Sigma_n(S)=\Sigma_n(S')$, where $S'\mid S$ is the subsequence with $\vp_x(S')=\min\{\vp_x(S),\,n\}$ for all $x\in \supp(S)$. In light of this basic observation, it generally makes little sense to consider $\Sigma_n(S)$ without the additional assumption limiting the maximal multiplicity to $\mathsf h(S)\leq n$. To a lesser extent, this also means that when studying $|\Sigma_{|G|}(S)|$, terms with multiplicity greater than $|S|-|G|$ are also redundant. Note, the coset condition for $S$ with the trivial subgroup is equivalent to $\mathsf h(S)\leq |S|-|G|+1$, so this is nearly achieved as part of the hypotheses of Theorem \ref{thm-olson}. When $|S|\geq 2|G|-1$, there can only be  one term $x\in \supp(S)$ with $\vp_x(S)=|S|-|G|+1$, meaning at most one term in $S$ is redundant, which proves to be negligible loss. However, the examples given above make use of much more non-negligible loss when $|S|=n+|G|$ with $n$ much smaller than $|G|-1$. If we disallow such redundant terms by imposing the slightly stronger hypothesis $\mathsf h(S)\leq |S|-|G|$, then we can obtain a result with optimal bounds for the size of $|S|$. The optimality of the bounds for $n$ can be seen by Examples B.1--B.3. Note $|S|=|G|+n$ with $\mathsf h(S)\leq n$ ensures that $n\geq 2$.

\begin{theorem}\label{thm-main-olson}
Let $G$ be a finite abelian group, let $n\geq 2$, and let $S\in \Fc(G)$ be a sequence of terms from $G$ with $|S|=|G|+n$ and $\mathsf h(S)\leq n$. Suppose, for every $H<G$ and $\alpha\in G$, there are at least $|G/H|-1$ terms of $S$ lying outside the coset $\alpha+H$. Then $\Sigma_{|G|}(S)=G$ whenever
\begin{itemize}
\item[1.] $n\geq \exp(G)$, or
\item[2.]  $n\geq \exp(G)-1$, \quad $G\cong H\oplus C_{\exp(G)}$, \quad and \quad either  $|H|$ or $\exp(G)$  is  prime, or
\item[3.] $n\geq \frac{|G|}{p}-1$ and $G$ is cyclic, where $p$ is the smallest prime divisor of $|G|$, or
\item[4.] $n\geq 2$ and either \quad $\exp(G)\leq 3$, \quad or \quad $|G|<12$, \quad or \quad $\exp(G)=4$ and $|G|= 16$.
\end{itemize}
\end{theorem}

\subsection*{Example B.1} Suppose $G=\la g\ra$ is cyclic of composite order $|G|\geq 10$ with $p$ the smallest prime divisor of $|G|$ and $H<G$ the subgroup of order $p$. Then the sequence $S'=\prod_{x\in H\cup (g+H)}^{\bullet}x^{[\frac1p |G|-2]}$ has $|S'|=2|G|-4p\geq |G|+\frac1p|G|-3$ with the inequality strict for $|G|>10$, $\mathsf h(S')\leq \frac1p|G|-2$, and $\Sigma_{\frac1p|G|-2}(S')\neq G$. If $S\mid S'$ is any subsequence with $|S|=|G|+\frac1p |G|-2$, then $\mathsf h(S)\leq \frac1p|G|-2$ and $\Sigma_{|G|}(S)=\sigma(S)-\Sigma_{\frac1p|G|-2}(S)\subseteq \sigma(S)-\Sigma_{\frac1p|G|-2}(S')\neq G$. If we choose $S$ so that $\vp_h(S)=\frac1p|G|-2$ for all $h\in H\cup \{g\}$, then $S$  (and also $S'$)  satisfies the coset condition. This shows the bound $n\geq \frac1p|G|-1$ is tight in Theorem \ref{thm-main-olson}.3 and in  Theorem \ref{thm-main-nsums}.3 below.

\subsection*{Example B.2} Suppose $G=H \oplus \la g\ra\cong H\oplus C_{\exp(G)}$ with $H$ nontrivial and $\exp(G)\geq 5$. In this case, the sequence $S'=\prod_{x\in H\cup (g+H)}^{\bullet}x^{[\exp(G)-2]}$ has $|S'|=2|G|-4|H|\geq |G|+\exp(G)-2$ and $\Sigma_{\exp(G)-2}(S')\neq G$.
If $S\mid S'$ is any subsequence with $|S|=|G|+\exp(G)-2$, then $\mathsf h(S)\leq \exp(G)-2$ and $\Sigma_{|G|}(S)=\sigma(S)-\Sigma_{\exp(G)-2}(S)\subseteq \sigma(S)-\Sigma_{\exp(G)-2}(S')\neq G$. If we choose $S$ so that $\vp_x(S)=\exp(G)-2$ for all $x\in H\cup \{g\}$, then $S$ (and also $S'$) satisfies the coset condition. This shows the bound $n\geq \exp(G)-1$ is tight in Theorem \ref{thm-main-olson}.2 and in Theorem \ref{thm-main-nsums}.2 below.

\subsection*{Example B.3} Suppose $G=H\oplus \la g\ra\cong H\oplus C_{\exp(G)}$ with $|H|$ and $\exp(G)$ composite, which implies $|G|\geq 16$. Let $K<H$ be a subgroup of order $|K|=\frac1p|H|$ where $p$ is the smallest prime divisor of $|H|$.  The sequence $S'=\prod_{x\in H\cup (g+K)}^{\bullet}x^{[\exp(G)-1]}$ has $|S'|=\frac{(p+1)}{p}|H|(\exp(G)-1)=|G|+\frac{m-p-1}{pm}|G|$, where $m=\exp(G)$, and $\Sigma_{\exp(G)-1}(S')\neq G$. Note $|G|=m|H|\geq mp^2$  since $|H|$ is composite, while $m=\exp(G)\geq 2p$ since $\exp(G)$ is composite. Thus, if $m-2\geq \frac{m-p-1}{pm}|G|$, then $m-2\geq (m-p-1)p$, implying $p^2-3p+2\leq 0$, in turn implying $p=2$. Moreover, $\frac{m-p-1}{pm}|G|\geq m-2$.
When $p=2$, we further obtain $m-2\geq (m-p-1)p=2m-6$, implying $\exp(G)=m\leq 4$, which is only possible if $m=\exp(G)=4$ (since $\exp(G)$ is composite). Hence $m-2\geq \frac{m-p-1}{pm}|G|$ implies equality holds with  $|G|\leq 16$ and $\exp(G)=4$. However, since $|H|$ is composite, this is only possible if $|G|=16$. Therefore, we conclude that $|S'|\geq |G|+\exp(G)-2$, and that  $|S'|\geq |G|+\exp(G)-1$ when  $|G|\neq 16$. In the latter case,
letting $S\mid S'$ be a subsequence of length $|S|=|G|+\exp(G)-1$, we find that $\mathsf h(S)\leq \exp(G)-1$ and  $\Sigma_{|G|}(S)=\sigma(S)-\Sigma_{\exp(G)-1}(S)\subseteq\sigma(S)-
\Sigma_{\exp(G)-1}(S')\neq G$. If we choose $S$ so that $\vp_x(S)=\exp(G)-1$ for all $x\in H\cup \{g\}$, then $S$  also satisfies the coset condition. This shows the bound $n\geq \exp(G)$ is tight in Theorem \ref{thm-main-olson}.1. Since $S'$ also satisfies the coset condition with $|S'|\geq |G|+\exp(G)-2$, the bound $n\geq \exp(G)$ is also tight in Theorem \ref{thm-main-nsums}.1 below.

\medskip

The second approach to generalizing Theorem \ref{thm-main-olson} is to replace $\Sigma_{|G|}(S)$ with $\Sigma_n(S)$ under the hypothesis that $\mathsf h(S)\leq n$. In this direction, there are results related to an analog of Kneser's Theorem for subsequence sums obtained either via the DeVos-Goddyn-Mohar Theorem \cite{DGM} \cite[Theorem 13.1]{Gbook} or the Partition Theorem \cite[Theorem 14.1]{Gbook}. Let us begin by stating the original theorem of Kneser for sumsets \cite[Theorem 4.1.1]{Alfred-Ruzsa-book} \cite[Theorem 5.2.6]{alfredbook} \cite[Theorem 6.1]{Gbook} \cite{kneserstheorem} \cite[Theorem 4.1]{natboook} \cite[Theorem 5.5]{taobook}.

\begin{theirtheorem}[Kneser's Theorem] Let $G$ be an abelian group and let $A_1,\ldots,A_n\subseteq G$ be finite, nonempty subsets. Then $$|\Sum{i=1}{n}A_i|\geq \Sum{i=1}{n}|A_i+H|-(n-1)|H|=\Sum{i=1}{n}|A_i|-(n-1)|H|+\rho,$$ where $H=\mathsf H(\Sum{i=1}{n}A_i)$ and $\rho:=\Sum{i=1}{n}|(A_i+H)\setminus A_i|$.
\end{theirtheorem}

Note $\Sum{i=1}{n}A_i=\Sum{i=1}{n}(A_i+H)$, and $\rho$ measures the number of ``holes'' in the sets $A_i$ relative to the sets $A_i+H$. The version of Kneser's Theorem valid for $n$-term subsums is the following (see the discussion in \cite[pp. 181--182]{Gbook}).

\begin{theirtheorem}[Subsum Kneser's Theorem]\label{thm-subsum-kneser}
Let $G$ be an abelian group, let $n\geq 1$, let $S\in \Fc(G)$ be a sequence  with $\mathsf h(S)\leq n\leq |S|$ and  let $H=\mathsf H(\Sigma_n(S))$. Then \begin{align*}|\Sigma_n(S)|\geq&\; (|S|-n+1)-(n-e-1)(|H|-1)+\rho,\\ =& \; |S|-(n-1)|H|+e(|H|-1)+\rho,\end{align*} where $\rho=|X||H|n+e-|S|\geq 0$, with
$X\subseteq G/H$ the subset of all $x\in G/H$ for which $x$ has multiplicity at least $n$ in $\phi_H(S)$, and  $e\geq 0$ the number of terms from $S$ not contained in $\phi_H^{-1}(X)$.
\end{theirtheorem}

 A short calculation shows that the bound given in Theorem \ref{thm-subsum-kneser} is equal to $((N-1)n+e+1)|H|=(\Summ{x\in G/H}\min\{n,\,\vp_x(\phi_H(S))\}-n+1)|H|$, where $N=|X|$, which is how the bound is stated in \cite{Gbook} and \cite{DGM}. The form given above is perhaps easier to apply in practice and highlights the connection with the bound from Kneser's Theorem better.  If we define $S^*$ to be the sequence obtained from $S$ (as given in Theorem \ref{thm-subsum-kneser}) by taking each term $x\in \phi_H^{-1}(X)$ and changing its multiplicity from $\vp_x(S)$ to $\vp_x(S^*)=n$, then $S\mid S^*$, $|S^*|=|S|+\rho$ and $\Sigma_n(S)=\Sigma_n(S^*)$ with $\rho$ measuring the number of ``holes'' in the sequence $S$ relative to $S^*$. The sequence $S^*$ plays the same role in Theorem \ref{thm-subsum-kneser} as the sets $A_i+H$ in the bound $|\Sum{i=1}{n}A_i|\geq \Sum{i=1}{n}|A_i+H|-(n-1)|H|$ obtained from Kneser's Theorem.
 %Theorem \ref{thm-subsum-kneser} (or earlier precursors)   has been used to resolve numerous conjectures in Zero-Sum Theory \cite{} and generalizes earlier results of Mann, Olson, and Bollob\'as and Leader.
As mentioned above, Theorem \ref{thm-subsum-kneser} can be obtained either from the DeVos-Goddyn-Mohar Theorem or the  Partition Theorem. The Partition Theorem first appeared (in some form) in \cite{ccd}, with the variation allowing $S'\mid S$ appearing in \cite{hamconj}. The more general form given below, which subtlety refines and strengthens the Subsum Kneser's Theorem, may be found in \cite[Theorem 14.1]{Gbook}, slightly reworded here.

\begin{theirtheorem}[Partition Theorem]\label{thm-partition-thm} Let $G$ be an abelian group, let $n\geq 1$, let $S\in \Fc(G)$ be a sequence, let $S'\mid S$ be a subsequence with $\mathsf h(S')\leq n\leq |S'|$, let $H=\mathsf H(\Sigma_n(S))$, let $X\subseteq G/H$ be the subset of all $x\in G/H$ for which $x$ has multiplicity at least $n$ in $\phi_H(S)$, and let $e$ be the number of terms from $S$ not contained in $\phi_H^{-1}(X)$.  Then there exists a setpartition $\sA=A_1\bdot\ldots\bdot A_n\in \sP(G)$ with $\mathsf S(\sA)\mid S$ and $|\mathsf S(\sA)|=|S'|$ such that
either
\begin{itemize}
\item[1.] $|\Sigma_n(S)|\geq |\Sum{i=1}{n}A_i|\geq \Sum{i=1}{n}|A_i|-n+1=|S'|-n+1$, or
\item[2.]
$|\Sigma_n(S)|=|\Sum{i=1}{n}A_i|\geq \Sum{i=1}{n}|A_i+H|-(n-1)|H|=|S'|-(n-1)|H|+e(|H|-1)+\rho$, where  $\rho=|X||H|n+e-|S'|\geq 0$, \quad $\Sigma_n(S)=\Sum{i=1}{n}A_i$ with $H$ nontrivial, \quad $|A_j\setminus \phi_H^{-1}(X)|\leq 1$ \quad  and  \quad $\supp(\mathsf S(\sA)^{[-1]}\bdot S)\subseteq \phi_H^{-1}(X)\subseteq A_j+H$ \quad for all $j\in [1,n]$.
\end{itemize}
\end{theirtheorem}

If $|\Sigma_n(S)|< |S|-n+1$, then combining this bound with the lower bound from Theorem \ref{thm-subsum-kneser} implies  that there are a small number of $H$-cosets, namely $N=|X|\geq 1$, containing most of the terms from $S$, namely all but $e$ terms. For large $n$, say $n\geq \frac1p|G|-1$ where $p$ is the smallest prime divisor of $\exp(G)$, comparing these lower and upper bounds forces $N=1$, leading to the coset condition holding for $S$, giving a version of Olson's Theorem valid for $n$-sums. However, it is actually possible to force the coset condition to hold for much smaller $n$. For instance, such a result was achieved for $n\geq \mathsf D^*(G)-1$ in \cite{oscar-weighted}. The Partition Theorem yields the bound given in Theorem \ref{thm-subsum-kneser}  but also shows that there is an actual setpartition $\sA=A_1\bdot\ldots\bdot A_n$ with either $|\Sigma_n(S)|\geq |\Sum{i=1}{n}A_i|\geq |S'|-n+1$ or $\Sigma_n(S)=\Sum{i=1}{n}A_i$. We will use this realization of $\Sigma_n(S)$ as a sumset together with the results from Section \ref{sec-itterated-sumsets} to reduce even further the necessary lower bound for $n$, and thereby obtain a generalization of Olson's Theorem \ref{thm-olson} from $|G|$-term to $n$-term subsums with optimal bounds for how large $n$ must be. The optimality follows in view of  Examples B.1--B.3.

\begin{theorem}\label{thm-main-nsums}
Let $G$ be a finite abelian group, let $n\geq 1$, and let $S\in \Fc(G)$ be a sequence of terms from $G$ with $|S|\geq n+|G|-1$ and $\mathsf h(S)\leq n$. Suppose, for every $H<G$ and $\alpha\in G$, there are at least $|G/H|-1$ terms of $S$ lying outside the coset $\alpha+H$. Then $\Sigma_n(G)=G$ whenever
\begin{itemize}
\item[1.] $n\geq \exp(G)$, or
\item[2.] $n\geq \exp(G)-1$, \quad   $G\cong H\oplus C_{\exp(G)}$, \quad and  \quad either  $|H|$ or $\exp(G)$ is prime, or
\item[3.] $n\geq \frac1p|G|-1$ and $G$ is cyclic,  where $p$ is the smallest prime divisor of $|G|$, or
\item[4.] $n\geq 1$ \quad and  \quad either $\exp(G)\leq 3$ or $|G|<10$.
\end{itemize}
\end{theorem}

\section{Critical Pair Theory}\label{sec-CriticalPair-Theory}

We review the portions of Kemperman's Critical Pair Theory  needed for the paper. We begin with the following simple consequence of the Pigeonhole Principle \cite[Theorem 5.1]{Gbook}. Note, if $A$ and $B$ are each subsets of an $H$-coset with $|A|+|B|\geq |H|+1$, then Theorem \ref{PigeonHoleBound} (applied to $A$ and $B$ translated so that they are subsets of the subgroup $H$) ensures that $A+B$ is an $H$-coset.

\begin{theirtheorem}[Pigeonhole Bound]\label{PigeonHoleBound} Let $G$ be an abelian group and let $A,\,B\subseteq G$ be finite subsets. If $|A|+|B|\geq |G|+r$ with $r\geq 1$ an integer, then $A+B=G$ with $\mathsf r_{A+B}(x)\geq r$ for every $x\in G$.
\end{theirtheorem}

Given a subgroup $H\leq G$, subset $A\subseteq G$ and $x\in A$,
we call the subset $(x+H)\cap A\neq \emptyset$ an \emph{$H$-coset slice} of $A$. The set $A$ naturally decomposes into the disjoint union of it's $H$-coset slices,   $A=A_{1}\cup\ldots\cup A_{d}$ with each $A_{i}=(x_i+H)\cap A$ for  some $x_i\in A$. We call such a decomposition the \emph{$H$-coset decomposition} of $A$. Thus $\phi_H(A)=\{\phi_H(x_1),\ldots,\phi_H(x_d)\}$ with the elements $\phi_H(x_i)$ distinct.
%If $\phi_H(x_{i+1})-\phi_H(x_i)$ is constant for $i=1,\ldots,d-1$, so $\phi_H(A)=\{\phi_H(x_1),\ldots,\phi_H(x_d)\} $ is an arithmetic progression with the indices chosen to reflect the order of terms in the progression, then we say $A=A_{1}\cup\ldots\cup A_{d}$ is an \emph{$H$-coset progression decomposition} of $A$.
If  $A=(A\setminus A_\emptyset)\cup A_\emptyset$ with  $A_\emptyset$  a nonempty subset of an $H$-coset and $A\setminus A_\emptyset$ $H$-periodic, then we call $A=(A\setminus A_\emptyset)\cup A_\emptyset$ an \emph{$H$-quasi-periodic decomposition} of $A$.
 %with $A_\emptyset$ the \emph{aperiodic part}.
 Note this means (after re-indexing the terms in its $H$-coset decomposition) that $A_{i}=x_i+H$ for $i\in [2,d]$ and $A_{1}=A_\emptyset$.
 %If $A\setminus A_\emptyset\neq \emptyset$, then we say that $A$ is \emph{$H$-quasi-periodic}, and if  $A$ is either periodic or  $H$-quasi-periodic with $H$ nontrivial, then we say that $A$ is \emph{quasi-periodic}.

Let $X,\,Y\subseteq G$ be finite and nonempty subsets with $H=\la X+Y\ra_*$.  We say that the pair $(X,Y)$ is \emph{elementary of  type} (I), (II), \ldots, (VII) or (VIII) if there are $z_A,\,z_B\in G$ such that $X=z_A+A$ and $Y=z_B+B$ for a pair of subsets $A,\,B\subseteq H$ satisfying the corresponding requirement below (with all complements relative to the subgroup $H$, so $\overline{A}=(H+A)\setminus A=H\setminus A$):
\begin{itemize}
\item[(I)]  $|A|=1$ or $|B|=1$
\item[(II)] $A$ and $B$ are arithmetic progressions of common difference $d\in H$ with $|A|,\,|B|\geq 2$ and $\ord(d)\geq |A|+|B|-1\geq 3$
\item[(III)] $|A|+|B|=|H|+1$ and there is precisely one unique expression element in the sumset $A+B$; in particular,   $A+B=H$
    \item[(IV)] $B=-\overline{A}$ and the sumset $A+B$ is aperiodic and contains no unique expression elements; in particular,   $A+B=A-\overline{ A}=H\setminus \{0\}$
        \item[(V)] $|A|=2$ or $|B|=2$, \quad and \quad $|A+B|=|A|+|B|$
        \item[(VI)]  $|A|=|B|=3$, \  $A=B$ \ and \ $|A+B|=|A|+|B|=6$
        \item[(VII)]  either $|A|=3$, $|2A|=6$, $B=-\overline{2A}$ and $A=-\overline{A+B}$, or else $|B|=3$, $|2B|=6$, $A=-\overline{2B}$  and $B=-\overline{A+B}$; in particular, $|A+B|=|A|+|B|=|H|-3$
\item[(VIII)]    there are subgroups $K_1,\,K_2,\,K< H$ with $K=K_1\oplus K_2\cong C_2\oplus C_2$ such that \begin{align*}&A=\Big(x_A+K_1\Big)\cup \Big (y+K\Big)\cup\ldots\cup \Big((r_A-1)y+K\Big)\cup \Big (r_Ay+K_2\Big) \und\\ &B=\Big(x_B+K_1\Big)\cup \Big (y+K\Big)\cup\ldots\cup \Big((r_B-1)y+K\Big)\cup \Big (r_By+K_2\Big)
                  \end{align*}
                  for some  $x_A,\,x_B\in K$, \ $y\in H\setminus K$, \ and $r_A,\,r_B\geq 1$ with $r_A+r_B+1\leq \ord(\phi_K(y))$; in particular, $|A+B|=|A|+|B|$
\end{itemize}

As is easily observed, $|A+B|=|A|+|B|-1$ when $(A,B)$ is an elementary pair of type (I), (II), (III) or (IV), and $|A+B|=|A|+|B|$ when $(A,B)$ is an elementary pair of type (V), (VI), (VII) or (VIII). These elementary pairs are the basic building blocks of all sumsets $A+B$ with $|A+B|\leq |A|+|B|$.

In view of Kneser's Theorem, the study of sumsets with $|A+B|\leq |A|+|B|$ reduces to the aperiodic case. This structure is fully characterized in \cite[Theorem 4.1, Corollary 4.2]{kst+1}. Combining this result with the ``dual'' formulation of the Kemperman Structure Theorem \cite[Theorem 9.2]{Gbook}, which characterizes the case when $|A+B|\leq |A|+|B|-1$, we can now summarize the relevant structural information we will need. We remark that the structural information given by Theorem \ref{KST-lemma} is fairly weak when $|A+B|\geq |G|-2$, though as a trade-off the sumset is quite large. We will be able to eliminate this case in Theorem \ref{thm-main-sumset} when passing to the iterated sumset $nA$, which will be a necessary step for deriving our generalization of Olson's result.

\begin{theirtheorem}\label{KST-lemma}
Let $G$ be a nontrivial abelian group and let $A,\,B\subseteq G$ be finite and nonempty subsets with $\la A+B\ra_*=G$. Suppose  $|A+B|\leq |A|+|B|$ and $A+B$ is  aperiodic. Then one of the following holds.
\begin{itemize}
\item[(i)] $(A,B)$ is an elementary pair of some type (I)--(II) or (IV)--(VIII).
\item[(ii)] $|A+B|=|A|+|B|\geq |G|-2$. Moreover, either $A-\overline{A+B}=-\overline {B}$ or  $A-\overline{A+B}=-\overline {B}\setminus \{x\}$ for some $x\in G$, where $\overline X=G\setminus X$, with the latter only possible if $|A+B|=|G|-1$.
\item[(iii)] There are arithmetic progressions $P_A,\,P_B\subseteq G$ of common difference such that $A\subseteq P_A$, \ $B\subseteq P_B$, and  $|P_A\setminus A|,\,|P_B\setminus B|\leq 1$.
\item[(iv)] There is a subgroup $H<G$ with $|H|=2$ such that \begin{align*}
&z_A+A=\{x_A\}\cup \Big(y+H\Big)\cup \ldots\cup \Big(r_Ay+H\Big)\cup \{(r_A+1)y\}\quad\und
\\ &z_B+B=\{x_B\}\cup \Big(y+H\Big)\cup \ldots\cup \Big(r_By+H\Big)\cup \{(r_B+1)y\},\end{align*}
for some   $z_A,\, z_B\in G$, \, $x_A,\,x_B\in H$, \  $y\in G\setminus H$ \ and \ $r_A,\,r_B\geq 1$ \ with \ $r_A+r_B+3\leq \ord(\phi_H(y))$.

\item[(v)]
 There is a nontrivial subgroup $H<G$  such that \begin{align*}&z_A+A=\{0\}\cup \Big(y+(H\setminus \{x\})\Big)\cup \Big(2y+H\Big)\cup \ldots\cup \Big(r_Ay+H\Big)\quad\und \\
 &z_B+B=\{0\}\cup \Big(y+(H\setminus \{x\})\Big)\cup \Big(2y+H\Big)\cup \ldots\cup \Big(r_By+H\Big),
 \end{align*} for some $z_A,\,z_B\in G$, \  $x\in H$, \  $y\in G\setminus H$ \ and \ $r_A,\,r_B\geq 1$  with $r+r'+1\leq \ord(\phi_H(y))$. Moreover,    $r_A,\,r_B\geq 2$ when $|H|=2$.
 \item[(vi)] There exists a proper, finite and nontrivial subgroup $H<G$ and nonempty subsets $A_\emptyset=(\alpha+H)\cap A$ and $B_\emptyset=(\beta+H)\cap B$, for some $\alpha,\,\beta\in G$, such that \begin{itemize}
    \item[(a)] $(\phi_H(A),\phi_H(B))$ is an elementary pair of some type (I)--(III),
     \item[(b)] $\phi_H(A_\emptyset)+\phi_H(B_\emptyset)$ is a unique expression element in $\phi_H(A)+\phi_H(B)$,
     \item[(c)] $|A\setminus A_\emptyset|\geq |H+(A\setminus A_\emptyset)|-1$ and $|B\setminus B_\emptyset|\geq |H+(B\setminus B_\emptyset)|-1$, with equality in either only possible when $|A+B|=|A|+|B|$,

    \item[(d)] $A_\emptyset+B_\emptyset$ is aperiodic with $-1\leq |A_\emptyset+B_\emptyset|-|A_\emptyset|-|B_\emptyset|\leq |A+B|-|A|-|B|$, and
\item[(e)]    $A+B=(A+B)\setminus (A_\emptyset+B_\emptyset)$ is  $H$-periodic.
\end{itemize}
         \end{itemize}
    %
    %
    %$H$-quasi-periodic decompositions $A=(A\setminus A_\emptyset)\cup A_\emptyset$ and $B=(B\setminus B_\emptyset)\cup B_\emptyset$  with $A_\emptyset$ and $B_\emptyset$ nonempty such that the pair $(\phi_H(A),\phi_H(B))$ is elementary of some type (I)--(IV), with both type (IV), as well as type (I) with $|\phi_H(A)|=|\phi_H(B)|=1$, only occurring when $H$ is trivial. Moreover, $\phi_H(A_\emptyset)+\phi_H(B_\emptyset)$ is either a unique expression element in $\phi_H(A)+\phi_H(B)$ or else $(\phi_H(A),\phi_H(B))=(A,B)$ has type (IV) with $H$ trivial.
    %\item[(ii)] $|A+B|=|A|+|B|$ and $(A,B)$ is of some type (V)--(VIII).
    %\item[(iii)] $|A+B|=|A|+|B|$ and there exists a type (II) pair $(A',B')$ such that $|A'|+|B'|\geq 5$, \ $A=A'\setminus \{a_2\}$ and $B=B'\setminus \{b_2\}$, where $a_2$ and $b_2$ are the respective second terms in $A'$ and $B'$.
     %   \item[(iv)] $|A+B|=|A|+|B|=|G|-2$.
      %  \item[(v)] $|A+B|=|A|+|B|$ and there exists a proper, finite and nontrivial subgroup $H<G$ and $H$-coset progression decompositions
       %           $A=A_1\uplus\ldots\uplus A_r$ and $B=B_1\uplus\ldots\uplus B_s$ such that $(\phi_H(A),\phi_H(B))$ is of type (II), \ $|A_r|=|B_s|=1$, \ $\Sum{i=1}{r-1}|A_i|=(r-1)|H|-1$ and $\Sum{i=1}{s-1}|B_i|=(s-1)|H|-1$.
        %\item[(vi)] $|A+B|=|A|+|B|$ and  there exists $\alpha,\,\beta\in G$ such that $(A\cup \{\alpha\})+(B\cup\{\beta\})=A+B$ and $|A\cup \{\alpha\}|+|B\cup\{\beta\}|=|A|+|B|+1$.
%\end{itemize}
\end{theirtheorem}

\section{Iterated Sumsets}\label{sec-itterated-sumsets}

The goal of this section is to derive improved structural information when  $\la A\ra_*=G$ and $|nA|<\min\{|G|,\,(|A|+1)n-3\}$ with $n\geq 3$.  The behaviour of $nA$ when $|A|\leq 2$ is rather straightforward, since in this case $A$ is an arithmetic progression. We begin with the first nontrivial case: $|A|=3$. We remark that most of the difficulty for Lemma \ref{lemma-itt-kst-size3} is dealing with the case when $|nA|=|G|-1$. Since $nA\neq G$ in these cases, having a \emph{structural} description of those sets $A$ which just fail to have maximal size sumset will be rather crucial for the later application generalizing Olson's result.

\begin{lemma}\label{lemma-itt-kst-size3}
Let $G$ be an  abelian group,  let $A\subseteq G$ be a subset with $\la A\ra_*=G$ and $|A|=3$, and let $n\geq 3$ be an integer. Suppose \be\label{bound-size3}|nA|<\min\{|G|,\,(|A|+1)n-3\}=\min\{|G|,\,4n-3\}.\ee
Then $nA$ is aperiodic and one of the following holds.
\begin{itemize}

\item[(i)] There is an arithmetic progression $P\subseteq G$ such that $A\subseteq P$ and  either
     \begin{itemize}
    \item[(a)] $3\leq |P|\leq 4$ (in which case   $|nA|=2n+1$, $|nA|=3n$ or $3n-1=|nA|=|G|-1$),
    \item[(b)] $|P|=5$ and  $4n-5\leq |nA|=|G|-1\leq 4n-4$, or
    \item[(c)] $|P|=6$ and $4n-4=|nA|=|G|-1=20$.
     \end{itemize}
\item[(ii)] There is an $H$-coset decomposition $A=\{x,\,z\}\cup \{y\}$ with $\la x-z\ra=H$ such that either
    \begin{itemize}
    \item[(a)] $2\leq |H|\leq 3$ (in which case $|nA|=2n+1$,  $|nA|=3n$ or $3n-1=|nA|=|G|-1$),
    \item[(b)] $|H|=4$ and $4n-5=|nA|=|G|-1$, or
    \item[(c)] $|H|=|G/H|=5$ and $4n-4=|nA|=|G|-1=24$.
    \end{itemize}
 \item[(iii)]  $G\cong C_2\oplus C_{\exp(G)}$  with $4\mid \exp(G)$ and  there is an $H$-coset decomposition $A=\{x,z\}\cup \{y\}$ with $\la x-z\ra=H$  such that $|G/H|=2$, \ $2(y+z)=4x$ \ and \ $4n-5=|nA|=|G|-1$.
\item[(iv)] $G\cong\Z/m\Z$, \ $8\nmid m$, \ $4n-5=|nA|=|G|-1$ \ and \ $A=\{0,1,\frac{m}{2}-1\}$ up to affine transformation.
 \end{itemize}
\end{lemma}

\begin{proof}
Suppose $nA$ is periodic, say with $H=\mathsf H(nA)$ nontrivial. Since $|nA|<|G|$, $H$ must be a proper subgroup. Thus, since $\la A\ra_*=G$, we have $2\leq |\phi_H(A)|\leq 3$.  Kneser's Theorem ensures that $|nA|\geq (n|\phi_H(A)|-n+1)|H|$. Thus, if $|\phi_H(A)|=3$, then the bound becomes $|nA|\geq (2n+1)|H|\geq 4n+2$, contrary to hypothesis. Therefore $|\phi_H(A)|=2$, implying we have an $H$-coset decomposition $A=A_1\cup A_0$ with $|A_0|=1$. Thus, since $\phi_H(A)$ is an arithmetic progression of size $2$ and $\la \phi_H(A)\ra_*=G/H$, the only way $nA$ can be $H$-periodic is if $n\geq |G/H|$, in which case  $|nA|\geq (n|\phi_H(A)|-n+1)|H|=(n+1)|H|>|G|$, which is not possible. So we instead conclude that $nA$ is aperiodic.

If (i)(a) holds with $|P|=3$, say w.l.o.g. $A=\{0,x,2x\}$, then $nA=\{0,x,\ldots, (2n)x\}$, so $|nA|=2n+1$ as $nA\neq G=\la A\ra_*$. If (i)(a) holds with $|P|=4$, say w.l.o.g. $A=\{0,x,3x\}$, then  $nA=\{0,x,\ldots, (3n-2)x\}\cup \{(3n)x\}$. Then, since $nA\neq G= \la A\ra_*$, we either have $(3n)x\neq 0$, in which case $|nA|=3n$, or else $(3n)x=0$, in which case $|nA|=|G|-1=3n-1$. If (ii)(a) holds with $|H|=2$, say w.l.o.g. $A=H\cup \{y\}$ with $\la \phi_H(y)\ra =G/H$, then $nA=H\cup \Big(y+H\Big)\cup \ldots \cup \Big((n-1)y+H\Big)\cup \{ny\}$, so $|nA|=n|H|+1=2n+1$. If (ii)(a) holds with $|H|=3$, say $A=(H\setminus \{x\})\cup \{y\}$, then $nA=H\cup \Big(y+H\Big)\cup \ldots\cup \Big( (n-2)y+H\Big)\cup \Big ((n-1)y+(H\setminus \{x\})\Big)\cup \{ny\}$. Then, since $nA\neq G= \la A\ra_*$, we either have $ny\notin H$, in which case $|nA|=n|H|=3n$, or else $ny\in H$, in which case $|nA|=|G|-1=n|H|-1=3n-1$.

We may assume by contradiction that neither (i)(a) nor  (ii)(b) hold for $A$.
Let $k\in [2,n]$. The pair $((k-1)A,A)$ cannot be an elementary of type (I) since $|(k-1)A|\geq |A|=3>1$. The pair $((k-1)A,A)$ cannot be elementary of type (II) as then (i)(a) holds for $A$. The pair $((k-1)A,A)$ cannot be elementary of type (III) as this contradicts that $nA$ is aperiodic. If $((k-1)A,A)$ is elementary of type (IV), then $|kA|=|(k-1)A|+|A|-1=|(k-1)A|+2=|G|-1$, which is only possible if $k=n\geq 3$, as otherwise Theorem \ref{PigeonHoleBound} implies that $G=(k+1)A\subseteq nA$, contradicting that $|nA|<|G|$. The pair $((k-1)A,A)$ cannot be an elementary of type (V) since $|(k-1)A|\geq |A|=3>2$. If $((k-1)A,A)$ is elementary of type (VI), then $|(k-1)A|=|A|=3$, which is only possible for $k=2$, as otherwise $|(k-1)A|\geq |2A|\geq 2|A|-1=5$, with  the latter inequality in view  of Kneser's Theorem since $nA$ is  aperiodic. If $((k-1)A,A)$ is elementary of type (VII), then $|kA|=|(k-1)A|+|A|=|G|-3$, which is only possible if $k\in [n-1,n]$ since otherwise Kneser's Theorem  implies $|nA|\geq |kA+A+A|\geq |kA|+2|A|-2=|G|-3+4=|G|+1$, which is impossible. The pair $((k-1)A,A)$ cannot be elementary of type (VIII) since $|A|=3<4$. If Theorem \ref{KST-lemma}(ii) holds for  $((k-1)A,A)$, then $|kA|=|G|-2=|(k-1)A|+|A|$, which is only possible for $k=n$ as otherwise Kneser's Theorem implies $|nA|\geq |kA+A|\geq |kA|+|A|-1=|G|$, contradicting that $nA$ is aperiodic.  Theorem \ref{KST-lemma}(iii) cannot hold for $((k-1)A,A)$ as then (i)(a) holds for $A$.
Theorem \ref{KST-lemma}(iv) cannot hold for $((k-1)A,A)$ as $|A|=3<4$. Theorem \ref{KST-lemma}(v) cannot hold  for $((k-1)A,A)$ as then (ii)(a) holds for $A$ in view of $|A|=3$. Theorem \ref{KST-lemma}(vi) cannot hold for $((k-1)A,A)$ as otherwise Theorem \ref{KST-lemma}(v)(c)  combined with $|A|=3$ implies that (ii)(a) holds for $A$.

In view of the possibilities listed above, Theorem \ref{KST-lemma} implies  that $|kA|\geq |(k-1)A|+|A|=|(k-1)A|+3$ for $k\in [2,n]$ unless $((k-1)A,A)$ is elementary of type (IV), in which case $|kA|=|(k-1)A|+|A|-1=|(k-1)A|+2=|G|-1$ and $k=n$.  In particular, $|2A|=2|A|=6$ (since $|2A|\leq \frac12|A|(|A|+1)$ is a trivial upper bound).
%
%Suppose there is an $H$-coset decomposition $A=A_1\cup A_0$ with  $|H|=3$. Then w.l.o.g $|A_1|=2$ and $|A_0|=1$, and since $|H|=3$ is prime, it follows that $\la A_1\ra_*=H$. It is now readily seen that $|kA|=|(k-1)A|+3$ for $k\in [2,n]$ unless $|kA|=|(k-1)A|+2=|G|-1$. Note the latter is only possible if $k=n$. Thus, by iterated application of this bound for $k=2,3,\ldots,n$, we find that $|nA|= 3n$ or $|nA|=3n-1=|G|-1$, and (v) follows, as desired. So we may assume no such $H$-coset decomposition exists.
%
Likewise,  $|kA|\geq |(k-1)A|+|A|+1=|(k-1)A|+4$ for $k\in [2,n]$ except when $((k-1)A,A)$ is elementary of type (IV), (VI) or (VII), or when $|(k-1)A+A|=|(k-1)A|+|A|\geq |G|-2$.  Type (VI) is only possible for $k=2$,  type (IV) and the fourth possibility can only occur for $k=n$, and type (VII) is only possible for $k\in [n-1,n]$ and implies $|kA|=|G|-3$.

For $k\geq 2$, let $\epsilon_k$ be the integer such that  $|kA|= |(k-1)A|+4+\epsilon_k$. Thus
\be\label{3card-form}|nA|=4n-1+\Sum{i=2}{n}\epsilon_i.\ee In view of the above work, we have \be\label{whitetrain}\epsilon_2=-1,\quad  \epsilon_k\geq 0\; \mbox{ for all $k\in [3,n-2]$},\quad  \epsilon_{n-1}\geq -1,\quad\und\quad\epsilon_n\geq -2.\ee Moreover, $\epsilon_n=-2$ is only possible if $|nA|=|G|-1$; and $\epsilon_{n-1}= -1$ with $n-1>2$ is only possible if $|(n-1)A|=|G|-3$, in which case $|nA|=|G|-1$ with $\epsilon_n=-2$ necessarily following in view of $|nA|<|G|$. It is now clear that the hypothesis $|nA|<4n-3$ is only possible if $|G|$ is finite with $\epsilon_n=-2$ and $|nA|=|G|-1$, which we now assume. Moreover, we must either have $|nA|=4n-4=|G|-1$ or $n|A|=4n-5=|G|-1$, ensuring that $$|G|\equiv 0\mbox{ or }1\mod 4.$$ In view of \eqref{3card-form} and \eqref{bound-size3}, we have $$\Sum{i=2}{n-1}\epsilon_i\leq -1.$$

Suppose there is an $H$-coset decomposition $A=A_1\cup A_0$. Then w.l.o.g $|A_1|=2$ and $|A_0|=1$. Moreover, $A=A_1\cup A_0$ is also a $\la A_1\ra_*$-coset decomposition, so we may w.l.o.g. assume $H=\la A_1\ra_*$. Since (ii)(a) is assumed not to hold, we must have $|H|\geq 4$. Suppose $|H|=4$. If $|G/H|=2$, then $|G|=2|H|=8$ and  $|2A|=6=|G|-2$, in which case $|3A|=|G|$ in view of Theorem \ref{PigeonHoleBound}, contrary to hypothesis. Therefore we may assume $|G/H|\geq 3$. But now it is clear that $|kA|= |(k-1)A|+4$ when $3\leq k\leq \frac14|G|-1$, that   $|kA|=|(k-1)A|+3=|G|-3$ for $k=\frac14|G|$, and that  $|kA|=|(k-1)A|+2=|G|-1$ for $k=\frac14|G|+1$, and thus (ii)(b) follows. So we may assume $|H|\geq 5$.
%Suppose there is an arithmetic progression $P\subseteq G$ such that $A\subseteq P$ and $|P|=5$. Then w.l.o.g. $G=\Z/m\Z$  and $A=\{0,1,4\}$ with $m\geq 9$ (since $|2A|=6$). Hence $kA=[0,4k-6]\cup [4k-4,4k-3]\cup \{4k\}$ with $|kA|=4k-2$ when $2\leq k<\frac14|G|=\frac14m$. If $4\mid m$, then $|kA|=4k-3$ for $k=\frac14m$ and $|kA|=4k-5=m-1$ for $k=\frac14m+1$. If $m\equiv -1$ or $-2$ modulo $4$, then $|kA|=4k-3\geq m-2$ for $k=\lceil\frac14m\rceil$ and $|kA|=m$ for $k>\lceil\frac14m\rceil$. Finally, if $m\equiv 1\mod 4$, then $|kA|=4k-4=m-1$ for $k=\frac{m+3}{4}$. Thus (iii) holds. So we may assume no such arithmetic progression exists.
 If we also have $|G/H|\geq 5$, then $\epsilon_2=-1$, $\epsilon_3=0$ and $\epsilon_4=1$. Moreover, $\epsilon_5\geq 1$ unless $|G/H|=|H|=5$. However, if $|G/H|=|H|=5$, then we instead have $|G|=|G/H||H|=25$,  $\epsilon_5=0$, $\epsilon_6=-1$, and $\epsilon_7=-2$ with $|7A|=|G|-1$. In this case, (ii)(c) follows in view of \eqref{3card-form}. On the other hand, if we instead have $\epsilon_5\geq 1$, then  \eqref{3card-form} and \eqref{whitetrain} ensure  that $|nA|\geq 4n-3$, contrary to hypothesis. So we now conclude that if there is an $H$-coset decomposition $A=A_1\cup A_0$, then $|H|\geq 5$ and $|G/H|\leq 4$. In particular, considering $H=\la a-b\ra$ with $a,\,b\in A$ gives (note $\la a-b\ra=\la A\ra_*=G$ if $c\in \la a-b\ra_*+b$) \be\ord(a-b)\geq\max\{5,|G|/4\}\quad\mbox{ for all distinct $a,\,b\in A$}\label{ringraft}.\ee

Suppose there is an $H$-coset decomposition $A=\{x,z\}\cup \{y\}$ with $\la x-z\ra=H\leq  G$ a subgroup such that $|G/H|=2$ and $2(y+z)=4x$. By translating by $-z$, we can w.l.o.g. assume $A=\{x,0\}\cup \{y\}$ with $2y=4x$. Since $|G/H|=2$, we must have $|G|$ even, whence $|nA|=4n-5=|G|-1$ as noted above, ensuring that $|G|$ is divisible by $4$. If $G$ were cyclic, then $2y=4x$ combined with $|G|\equiv 0\mod 4$ and $|G/H|=2$ ensures that $y\in \la x\ra=H$, in which case $\la A\ra_*=\la 0,y,x\ra=\la x\ra=H<G$, contradicting the hypothesis $\la A\ra_*=G$. Therefore, since $\la x\ra=H$ is an index $2$ subgroup, we must have $G\cong C_2\times C_{\exp(G)}$. It remains to show $4\mid \exp(G)$, which in view of $4(n-1)=|G|=2\exp(G)$ is equivalent to $n$ being odd, and then (iii) will follow. To see this, we have only to note that \begin{align*}&kA=\{0,x,\ldots,(2k-1)x\}\cup y+\{0,x,\ldots,(2k-4)x,(2k-2)x\}&&\mbox{for $k\geq 3$ odd and }\\
&kA=\{0,x,\ldots,(2k-2)x,(2k)x\}\cup y+\{0,x,\ldots,(2k-3)x\}&&\mbox{for $k\geq 2$ even}.\end{align*} Consequently, since $|nA|=|G|-1$, we must either have $n$ odd with  $2n-2=|H|=\frac12|G|$ (note $|H|=\frac12|G|$ is even since $4$ divides $|G|$) or else $n$ is even with $2n=|H|=\frac12|G|$. Since $|nA|=4n-5=|G|-1$, only the former is possible, and (iii) follows. So we can now assume no such $H$-coset decomposition exists.

By translating $A$ appropriately, we can w.l.o.g. assume $0\in A$,  in which case $\la A\ra=\la A\ra_*=G$ ensures that $G$ is generated by the two non-zero elements of $A=\{0,a,b\}$. Thus $G$ has rank at most $2$.

\subsection*{Case A} $G$ is noncyclic, say   $G\cong \Z/m'\Z\times \Z/m\Z$ with $2\leq m'\mid m$.

Suppose both nonzero elements $a,\,b\in A$ have order less than $\exp(G)=m$. Then we have $\ord(a),\,\ord(b)\leq \frac{|G|}{2m'}$, and we conclude that $m'=2$ and $\ord(a)=\ord(b)=|G|/4=m/2$ in view of \eqref{ringraft}. However, since any element of order $m/2$ in $G=\Z/2\Z\times \Z/m\Z$ must have an even second coordinate, this contradicts that $\la a,b\ra=\la A\ra=G$. So we can instead assume some element in $A$ has order equal to $\exp(G)=m$.

Any element $g\in G$ with $\ord(g)=\exp(G)$ generates a subgroup which is a direct summand in $G$. Thus we can w.l.o.g. assume $A=\{(0,0), (0,1), (1,x)\}$ with $0\leq x\leq \frac 12m$. In view of \eqref{ringraft} applied with $a=(0,1)$ and $b=(0,0)$, we conclude that $m=\exp(G)\geq 5$ and $m'\leq 4$. Thus $m\geq 6$ (since $m'\mid m$ with $m'\in [2,4]$ and $m\geq 5$), and since $|G|\equiv 0$ or $1$ modulo $4$, we conclude that either $m$ is even,  or else $m\equiv -1\mod 4$ with $m'=3$.
In view of \eqref{ringraft} applied with $a=(1,x)$ and $b=(0,0)$, we conclude that  $\ord((1,x))=m$ or $\frac{m}{2}$, with $\ord((1,x))=\frac{m}{2}$ only possible if $m'=2$. Likewise applying \eqref{ringraft} with $a=(1,x)$ and $b=(0,1)$, we conclude that $\ord((1,x-1))=m$ or $\frac{m}{2}$, with the latter only possible when $m'=2$.

Now $A\subseteq 2A\subseteq 3A\subseteq 4A$ with \begin{align*}4A=\{(0,0),(0,1),(0,2),(0,3),(0,4)\}\cup \{(1,x),(1,x+1),(1,x+2),(1,x+3)\}\cup\\ \{(2,2x),(2,2x+1),(2,2x+2)\}\cup \{(3,3x),(3,3x+1)\}\cup\{(4,4x)\}.\end{align*} Thus $|4A|\leq 15$, and in view of $m\geq 6$, it follows that the five groupings of elements above each consist of distinct elements. Consequently, if $|4A|<15$, then we  must have $2x\equiv -2,-1,0,1,2,3\mbox{ or } 4\mod m$ with $m'=2$, or $3x\equiv -1,0,1,2,3\mbox{ or } 4\mod m$ with $m'=3$, or $4x\equiv 0,1,2,3\mbox{ or } 4\mod m$ with $m'=4$ or $2$.
 Since $m'\mid m$ and $0\leq x\leq \frac{m}{2}$, it follows that either $x=0,1,2,\frac{m}{2},\frac{m}{2}-1,\frac{m}{4},\frac{m}{4}+1\mbox{ or } \frac{m+2}{4}$ with $m'=2$, or $x=0,1,\frac{m}{3},\frac{m}{3}+1$ with $m'=3$, or $x=0,\frac{m}{4},\frac{m}{2},1\mbox{ or }\frac{m}{4}+1$ with $m'=4$. When  $m'=4$, $x=0,\frac{m}{4}\mbox{ or } \frac{m}{2}$ implies $\ord((1,x))\leq 4<m$, while $x=1\mbox{ or }\frac{m}{4}+1$ implies $\ord((1,x-1))\leq 4<m$, both contradictions to what was shown above. When $m'=3$, $x=0\mbox{ or }\frac{m}{3}$ implies $\ord((1,x))=3<m$, while $x=1\mbox{ or } \frac{m}{3}+1$ implies $\ord((1,x-1))=3<m$, both contradictions to what was shown above.
 When $m'=2$, $x=0$ or $\frac{m}{2}$ implies $\ord((1,x))=2<3\leq \frac{m}{2}$; $x=1$ implies $\ord((1,x-1))=2<3\leq \frac{m}{2}$; $x=\frac{m}{4}$ implies either $x=2$ or else  $m\geq 12$ and $\ord((1,x))=4<6\leq \frac{m}{2}$; and $x=\frac{m}{4}+1$ implies either $x=\frac{m}{2}-1$ or else $m\geq 12$ and $\ord((1,x-1))=4<6\leq \frac{m}{2}$. Thus we obtain contradictions in all cases except for $m'=2$ with $x=2,\frac{m}{2}-1\mbox{ or }\frac{m+2}{4}$. If $x=\frac{m+2}{4}$, then we find that \begin{align*}&kA=\{(0,0),(0,1),\ldots,(0,k)\}\cup \{(0,\frac{m}{2}+1),\ldots,(0,\frac{m}{2}+k-1)\}\cup\\ &\{(1,\frac{m+2}{4}),\ldots,(1,\frac{m+2}{4}+k-1)\}\cup \{(1,\frac{3m+6}{4}),\ldots,(1,\frac{3m+6}{4}+k-3)\}\end{align*}
 for $k\geq 3$. Thus $|kA|=4k-2$ for $k\in [3,\frac{m}{2}]$, while $kA=G$ for $k=\frac{m}{2}+1$, contrary to \eqref{bound-size3}. If $x=2$, then $2(1,x)=(0,4)=4(0,1)$. Likewise, if $x=\frac{m}{2}-1$, then translating all terms by $(0,-1)$ yields $A=\{(0,0),(0,-1),(1,\frac{m}{2}-2)\}$ with $2(1,\frac{m}{2}-2)=(0,-4)=4(0,-1)$. Thus $A$ has an $H$-coset decomposition satisfying the requirements of (iii), yielding the full conclusion contained in (iii) as shown earlier. So we can now assume $|4A|=15$, and thus also that $|2A|=6$ and $|3A|=10$, since $2A\subseteq 3A\subseteq 4A$ with $4A$ only able to achieve its maximal  value $15$ if $|2A|$ and $|3A|$ also achieve their maximal values. It follows that $\epsilon_2=-1$, $\epsilon_3=0$ and $\epsilon_4=1$. In consequence, in view of \eqref{3card-form} and \eqref{whitetrain}, we find that we must have $n\geq 6$ with $|nA|=4n-4=|G|-1$, forcing $m'=3$. Moreover, we must have $n_5\leq 0$, as otherwise  \eqref{3card-form} and \eqref{whitetrain} ensure that $|nA|\geq 4n-3$, contrary to hypothesis.

 Since $n_5\leq 0$ with $|4A|=15$, we conclude that $|5A|\leq 19$. Now
\begin{align*}
&5A=\{(0,0),(0,1),(0,2),(0,3),(0,4),(0,5)\}\cup \{(1,x),(1,x+1),(1,x+2),(1,x+3),(1,x+4)\}\\&\cup \{(2,2x),(2,2x+1),(2,2x+2),(2,2x+3)\}\\&\cup \{(3,3x),(3,3x+1),(3,3x+2)\}\cup\{(4,4x),(4,4x+1)\}\cup \{(5,5x)\}.
\end{align*}
If all the elements listed above are distinct, then $|5A|=21$, contrary to what we concluded above. Thus, in view of $m=\exp(G)\geq 6$ and $m'=3$, we must have $3x\equiv -2,-1,0,1,2,3,4\mbox{ or }5\mod m$. Since $3=m'\mid m$, this is only possible if $3x\equiv 0\mbox{ or } 3\mod m$, both of which would contradict that $|4A|=15$. This completes the case when $G$ is non-cyclic.

\subsection*{Case B} $G=\Z/m\Z$ is cyclic.

Note $m=|G|\equiv 0\mbox{ or }1 \mod 4$ and $6=|2A|\leq |G|-3=m-3$ (the upper bound follows else Theorem \ref{PigeonHoleBound} implies that $3A=G$), ensuring that $m\geq 9$.

\subsection*{Case B.1} There is some generating element for $G$ contained in $A-A$.

In this subcase, by applying an appropriate affine transformation, we can w.l.o.g. assume $A=\{0,1,y\}$ with $2\leq y\leq \frac{m}{2}$ (if $y=\frac{m+1}{2}$, then $m$ is odd and the affine transformation $x\mapsto 2x$ yields $A=\{0,1,2\}$; otherwise, apply the affine transformation $x\mapsto -x+1$ when $y\geq \frac{m}{2}+1$). Moreover, since we have assumed (i)(a) does not hold, we must have $4\leq y\leq \frac{m}{2}$.
Now $A\subseteq 2A\subseteq 3A\subseteq 4A$ with \begin{align*}4A=[0,4]\cup (y+[0,3])\cup (2y+[0,2])\cup (3y+[0,1])\cup \{4y\}.\end{align*}
Thus $|4A|\leq 15$, and in view of $m\geq 9$, it follows that the five groupings of elements above each consist of distinct elements.
Consequently, if $|4A|<15$, then we  must have $y\equiv t\mod m$ for some $t\in [-3,4]$, or $2y\equiv t\mod m$ for some $t\in [-2,4]$, or $3y\equiv t\mod m$ for some $t\in [-1,4]$, or $4y\equiv t\mod m$ for some $t\in [0,4]$. Hence, since $m\equiv 0\mbox{ or } 1\mod 4$ with $4\leq y\leq \frac{m}{2}$ and $m\geq 9$, we conclude that $y\in \{4,\frac{m}{2}-1,\frac{m-1}{2},\frac{m}{2},\frac{m-1}{3},\frac{m}{3},
\frac{m+1}{3},\frac{m+2}{3},\frac{m}{3}+1,\frac{m+4}{3},\frac{m}{4},
\frac{m+3}{4},\frac{m}{4}+1\}.$ If $y=4$, then $P=\{0,1,2,3,4\}$ shows that  (i)(b) holds; if $y=\frac{m-1}{2}$, then (i)(a) holds with $P=\{\frac{m-1}{2},0,\frac{m+1}{2},1\}$; if $y=\frac{m}{2}$, then $A=\{0,\frac{m}{2}\}\cup \{1\}$ is an $H$-coset decomposition satisfying (ii)(a); if $y=\frac{m-1}{3}$, then $P=\{\frac{m-1}{3},0,\frac{2m+1}{3},\frac{m+2}{3},1\}$ shows that (i)(b) holds; if $y=\frac{m}{3}$, then $A=\{0,\frac{m}{3}\}\cup \{1\}$ is an $H$-coset decomposition satisfying (ii)(a); if $y=\frac{m+1}{3}$, then $P=\{0,\frac{m+1}{3},\frac{2m+2}{3},1\}$ shows that (i)(a) holds; if $y=\frac{m+2}{3}$, then $P=\{1,\frac{m+2}{3},\frac{2m+1}{3},0\}$ shows that (i)(a) holds; if $y=\frac{m}{3}+1$, then $A=\{1,\frac{m}{3}+1\}\cup \{0\}$ is an $H$-coset decomposition satisfying (ii)(a); if $y=\frac{m+4}{3}$, then $P=\{0,\frac{m+1}{3},\frac{2m+2}{3},1,\frac{m+4}{3}\}$ shows that (ii)(b) holds; if $y=\frac{m}{4}$, then $A=\{0,\frac{m}{4}\}\cup \{1\}$ is an $H$-coset decomposition satisfying (ii)(a); if $y=\frac{m+3}{4}$, then $P=\{1,\frac{m+3}{4},\frac{m+1}{2},\frac{3m+1}{4},0\}$ shows that (i)(b) holds; and if $y=\frac{m}{4}+1$, then $A=\{1,\frac{m}{4}+1\}\cup \{0\}$ is an $H$-coset decomposition satisfying (ii)(b).
Thus, in all cases except $y=\frac{m}{2}-1$, one of our desired conclusions follows. However, if $y=\frac{m}{2}-1$, then $m=|G|$ is even and $|nA|=4n-5=|G|-1=m-1$, whence $n=\frac{m}{4}+1$. In this case, \begin{align*}&kA=\{-k\}\cup [-k+2,k]\cup [\frac{m}{2}-k+1,\frac{m}{2}+k-2] \mod m &&\mbox{for $k\geq 2$ even, and}\\
 &kA=[-k+1,k]\cup \{\frac{m}{2}-k\}\cup [\frac{m}{2}-k+2,\frac{m}{2}+k-2]\mod m &&\mbox{for $k\geq 3$ odd}.\end{align*} Consequently, $nA=G$ if $n=\frac{m}{4}+1$ is odd, contrary to hypothesis, while if $n=\frac{m}{4}+1$ is even, then $8\nmid m=|G|$ and (iv) follows.
So we can instead assume $|4A|=15$, and thus also $|2A|=6$ and $|3A|=10$. It follows that $\epsilon_2=-1$, $\epsilon_3=0$ and $\epsilon_4=1$. In consequence, in view of \eqref{3card-form} and \eqref{whitetrain}, we find that we must have $n\geq 6$ with $|nA|=4n-4=|G|-1=m-1$, forcing $m\equiv 1\mod 4$ and $m=|G|\geq 21$. Moreover, we must have $n_5\leq 0$, as otherwise  \eqref{3card-form} and \eqref{whitetrain} ensure that $|nA|\geq 4n-3$, contrary to hypothesis.

 Since $n_5\leq 0$ with $|4A|=15$, we conclude that $|5A|\leq 19$. Now
\begin{align*}5A=[0,5]\cup (y+[0,4])\cup (2y+[0,3])\cup (3y+[0,2])\cup \{4y+[0,1]\}\cup\{5y\}.\end{align*}
If all the elements listed above are distinct apart from possibly $5y$, then $|5A|\geq 20$, contrary to what we concluded above. Thus, in view of $m\geq 21$, we  must have $y\equiv t\mod m$ for some $t\in [-4,5]$, or $2y\equiv t\mod m$ for some $t\in [-3,5]$, or $3y\equiv t\mod m$ for some $t\in [-2,5]$, or $4y\equiv t\mod m$ for some $t\in [-1,5]$. However, since $|4A|=15$, we can eliminate all the possibilities for $t$ considered in the previous paragraph, leaving only the following: $y\equiv t\mod m$ for some $t\in \{-4\}\cup \{5\}$, or $2y\equiv t\mod m$ for some $t\in \{-3\}\cup \{5\}$, or $3y\equiv t\mod m$ for some $t\in \{-2\}\cup \{5\}$, or $4y\equiv t\mod m$ for some $t\in \{-1\}\cup \{5\}$. Consequently, since $m\equiv 1\mod 4$ with $4\leq y\leq \frac{m}{2}$ and $m\geq 21$, we conclude that $y\in \{5,\frac{m-3}{2},\frac{m-2}{3},\frac{m+5}{3},\frac{m-1}{4}\}$. If $y=5$, then $5A=[0,13]\cup [15,17]\cup [20,21]\cup \{25\}$, in which case $|5A|=20$, contrary to assumption, unless $m\leq 25$ with $m\equiv 1\mod 4$.
If, in addition, $m=25$, then (ii)(c) holds, while  if in addition $m=21$, then (i)(c) holds.
If $y=\frac{m-3}{2}$, then applying the affine transformation $z\mapsto 2z+3$, we can assume $A=\{0,3,5\}$, in which case $5A=\{0,3,5,6\}\cup [8,20]\cup \{21,23,25\}$. Thus $|5A|=20$, contrary to assumption, unless $m\leq 25$ with $m\equiv 1\mod 4$.
If, in addition, $m=25$, then (ii)(c) holds, while  if in addition $m=21$, then (i)(c) holds.
If $y=\frac{m-2}{3}$, then applying the affine transformation $z\mapsto -3z+3$, we can assume $A=\{0,3,5\}$, and the argument is identical to the previous case when $y=\frac{m-3}{2}$.
If $y=\frac{m+5}{3}$, then applying the affine transformation $z\mapsto 3z$, we can assume $A=\{0,3,5\}$, and the argument is identical to the previous case once more.
Finally, if $y=\frac{m-1}{4}$, then applying the affine transformation $z\mapsto 4z+1$, we can assume $A=\{0,1,5\}$, in which the case the argument is identical to the case when $y=5$, completing subcase B.1.

\subsection*{Case B.2} $A-A$ contains no generating element for $G=\Z/m\Z$.

Let $A=\{0,x,y\}$ with $x,\,y\in [1,m-1]$. Since $A-A$ contains no generating element, it follows that $\gcd(x,m)\geq 2$, $\gcd(y,m)\geq 2$, $\gcd(x-y,m)\geq 2$. Since $\la A\ra=G$, we have $\gcd(x,y,m)=1$. Since every $H$-coset decomposition $A=A_1\cup A_0$  has $|H|\geq 5$ and $|G/H|\leq 4$, we must have $\gcd(x,m)\leq 4$, $\gcd(y,m)\leq 4$ and $\gcd(x-y,m)\leq 4$. Thus $\gcd(x,m),\,\gcd(y,m),\,\gcd(x-y,m)\in [2,4]$. If both $\gcd(x,m)\in [2,4]$ and $\gcd(y,m)\in [2,4]$ are even or are both equal to $3$, then this contradicts that  $\gcd(x,y,m)=1$. Thus $6\mid m$ and we may w.l.o.g. assume $\gcd(y,m)=3$ and $\gcd(x,m)\in \{2,4\}$. Let $x=2^sr$ with $r$ odd, $s\geq 1$ and $\gcd(r,m)=1$, and let $y=3t$ with $\gcd(t,\frac{m}{3})=1$. Then $x-y=2^sr-3t$ is neither divisible by $2$ nor $3$, contradicting that $\gcd(x-y,m)\in [2,4]$, which completes the subcase and proof.
\end{proof}

We note that most of the possibilities for $A$ given by Lemma \ref{lemma-itt-kst-size3} require $G$ to be finite with $\gcd(|G|,30)\neq 1$, the only exceptions being those in (i))(a) which ensure $A\subseteq P$ with $P$ a short length arithmetic progression (in which case $G$ is cyclic). Also, if $|nA|\leq 4n-6$, then  $|nA|\leq 3n$, showing there is a gap in the possible cardinalities for $|nA|$. Indeed, we always have $$\frac{|nA|}{n}\in \{2,3,4\}+\{-1/n, 0, 1/n, -4/n, -5/n\}$$ when $|A|<4n-3$. Conditions (i)(b), (i)(c), (ii)(b), (ii)(c), (iii) and (iv) each require both $|nA|=|G|-1$ and $|nA|\geq 4n-5$. Thus, if \eqref{bound-size3} is weakened to either $|nA|<\min\{|G|-1,4n-3\}$ or to $|nA|<\min\{|G|,4n-5\}$, then only conclusions (i)(a) or (ii)(a) can hold. Conclusions (i)(b), (i)(c), (ii)(b), (ii)(c), (iii) and (iv) also all require $|G|\equiv 0\mbox{ or } 1\mod 4$, and can be eliminated for $G$ infinite or with $|G|\equiv 2\mbox{ or } 3\mod 4$, where only conclusion (i)(a) or (ii)(a) can hold.

\begin{proof}[Proof of Theorem \ref{thm-main-sumset}]
We begin by calculating the size of $|nA|$ under each of the structural conditions given by Theorem \ref{thm-main-sumset}. Note, since $nA$ is aperiodic, we have $|nA|<|G|$. If Theorem \ref{thm-main-sumset}(i) holds with $A=P$, say w.l.o.g. $A=\{0,x,\ldots,(|A|-1)x\}$, then $nA=\{0,1,\ldots, n(|A|-1)x\}$ and $|nA|=(|A|-1)n+1$.  Suppose Theorem \ref{thm-main-sumset}(i) holds with $|P|=|A|+1$, say w.l.o.g. $P=\{0,x,\ldots,|A|x\}$ and $A=P\setminus \{rx\}$ with $r\in [ 1,|A|-1]$. If $r\in [2,|A|-2]$, then $A+A=\{0,x,\ldots,(2|A|)x\}$, \ $nA=\{0,x,\ldots,(n|A|)x\}$ and $|nA|=|A|n+1$. If w.l.o.g. $r=|A|-1$ (the case $r=1$ follows by replacing $x$ by $-x$ for the difference in the arithmetic progression $P$), then $nA=\{0,x,\ldots,(n|A|-2)x\}\cup \{(n|A|)x\}$. Thus either $(n|A|)x=0$ and $|nA|=|G|-1=|A|n-1$, or else $(n|A|)x\neq 0$ and $|nA|=|A|n$.
Instead suppose Theorem \ref{thm-main-sumset}(ii) holds, say w.l.o.g. $A=(x+K_1)\cup \Big(y+H\Big)\cup \ldots\cup \Big((r-1)y+H\Big)\cup \Big(ry+K_2\Big)$ with $r\geq 1$ and $H=K_1\oplus K_2\cong C_2\oplus C_2$. Then $|A|=4r$ and $nA=(nx+K_1)\cup \Big(y+H\Big)\cup \ldots\cup \Big((nr-1)y+H\Big)\cup \Big(nry+K_2\Big)$. Thus either $nry\notin H$, in which case $|nA|=4nr=|A|n$, or else  $nry\in H$, in which case $|nA|=|G|-1=4nr-1=|A|n-1$ since $|K_1\cup (z+K_2)|=3$ for any $z\in H$.
Instead suppose Theorem \ref{thm-main-sumset}(iii) holds, say w.l.o.g. $A=\{x\}\cup \Big (y+H\Big)\cup \ldots\cup \Big( ry+H\Big)\cup \{(r+1)y\}$ with $|H|=2$ and $r\geq 1$. Then $nA=\{nx\}\cup \Big (y+H\Big)\cup \ldots\cup \Big( (n(r+1)-1)y+H\Big)\cup \{n(r+1)y\}$ and $|A|=r|H|+2=2r+2$. If $n(r+1)y+H\neq H$, then $|nA|=(n(r+1)-1)|H|+2=2n(r+1)=|A|n$. If $n(r+1)y+H=H$, then we must have $n(r+1)y= nx$ in view of $|nA|<|G|$ and $|H|=2$, whence $|nA|=|G|-1=(n(r+1)-1)|H|+1=|A|n-1$. Next, suppose Theorem \ref{thm-main-sumset}(iv) holds, say w.l.o.g. $A=\{0\}\cup \Big (y+(H\setminus \{x\})\Big)\cup (2y+H)\cup \ldots\cup (ry+H)$ with $r\geq 1$. Furthermore, $r\geq 2$ when $|H|=2$. Then $|A|=r|H|$ and $nA=\{0\}\cup \Big (y+(H\setminus \{x\})\Big)\cup (2y+H)\cup \ldots\cup (rny+H)$. If $rny+H\neq H$, then $|nA|=rn|H|=n|A|$ in view of $|nA|<|G|$. If $rny+H=H$, then $|nA|=|G|-1=rn|H|-1=n|A|-1$. The size of $nA$ when Theorem \ref{thm-main-sumset}(v) holds will be calculated later.

Since $G$ is nontrivial with $\la A\ra_*=G$, we have $|A|\geq 2$. If $|A|=2$, then (i) holds with $|P|=|A|=2$. If $|A|=3$, then Lemma \ref{lemma-itt-kst-size3} completes the proof. Therefore we may assume $|A|\geq 4$ and w.l.o.g. (by translation) that $0\in A$. Since $nA$ is aperiodic, Kneser's Theorem implies that $|nA|\geq |(n-1)A|+|A|-1$. If $|kA|\geq |(k-1)A|+|A|+1$ for all $k\in [2,n-1]$, then $|(n-1)A|\geq (n-1)|A|+(n-2)$ follows by iterating these bounds, and then $|nA|\geq |(n-1)A|+|A|-1\geq |A|n+n-3$ follows, contrary to hypothesis. Therefore there is some $k\in [2,n-1]$ such that $|kA|\leq |(k-1)A|+|A|$, and we can apply Theorem \ref{KST-lemma} to $(k-1)A+A$.

If $|kA|\geq |G|-3$, then $|A|\geq 4$ combined with Theorem \ref{PigeonHoleBound} ensures that $(k+1)A=G$, whence $nA=G$ in view of $k<n$, contradicting that $nA$ is aperiodic with $G$ nontrivial. Therefore we can assume $|kA|\leq |G|-4$.
Since $|A|\geq 4$ and $|kA|\leq |G|-4$, $((k-1)A,A)$ cannot be elementary of type (I), (IV), (V), (VI) or (VII). If $((k-1)A,A)$ is elementary of type (II), then (i) follows. If $((k-1)A,A)$ is elementary of type (III), then $kA$ is periodic, and thus also $nA$ (as $k\leq n$), contrary to hypothesis. If $((k-1)A,A)$ is elementary of type (VIII), then (ii) follows. Therefore we can assume $((k-1)A, A)$ is not an elementary pair. If there is an arithmetic progression $P\subseteq G$ such that $A\subseteq P$ with $|P|\leq |A|+1$, then (i) follows. If Theorem \ref{KST-lemma}(iv) holds, then (iii) follows, while if Theorem \ref{KST-lemma}(v) holds, then (iv) follows.  Thus Theorem \ref{KST-lemma}(vi) must hold for $((k-1)A,A)$.

Let $H<G$ be a finite, nontrivial, proper subgroup such that Theorem \ref{KST-lemma}(vi) holds with $A_\emptyset=(x_0+H)\cap A$ and $B_\emptyset=(y_0+H)\cap B$, where $B=(k-1)A$. If $(\phi_H(A),\phi_H(B))$ is elementary of type (I), then this implies that $A$ is contained in an $H$-coset (since $A\subseteq (k-1)A=B$). However, in view of the hypothesis $\la A\ra_*=G$, this is only possible if $H=G$, contradicting that $H<G$ is proper. If $(\phi_H(A),\phi_H(B))$ is elementary of type (III), then $\phi_H(kA)=G/H$ and  Theorem \ref{KST-lemma}(vi)(e) ensures that $A+B=\Big((A+B)\setminus (A_\emptyset+B_\emptyset)\Big)\cup (A_\emptyset+B_\emptyset)$ is an $H$-quasi-periodic decomposition. In consequence, since $|\phi_H(A)|\geq 2$, it follows that $(k+1)A=G$, and thus $nA=G$ follows in view of $k<n$, contradicting that $nA$ is aperiodic with $G$ nontrivial. Therefore we must have $(\phi_H(A),\phi_H(B))$ elementary of type (II) by Theorem \ref{KST-lemma}(vi)(a). Moreover, since Theorem \ref{KST-lemma}(vi)(b) ensures that $\phi_H(A_\emptyset)+\phi_H(B_\emptyset)$ is a unique expression element in $\phi_H(A)+\phi_H(B)$, we must have $\phi_H(A_\emptyset)$ and $\phi_H(B_\emptyset)$ being the first term in the arithmetic progressions $\phi_H(A)$ and $\phi_H(B)$. Translating $A$ so that $0\in A_0=A_\emptyset$, we find $A_0\subseteq A\subseteq P:=A_0\cup (y+H)\cup \ldots\cup (ry+H)$, for some $y\in G\setminus H$, with $|P|=|A|+\epsilon$ for some $\epsilon \in \{0,1\}$ in view of Theorem \ref{KST-lemma}(vi)(c). Note  $r=|\phi_H(A)|-1\geq 1$, meaning (a) holds.
Since $A+B=\Big((A+B)\setminus (A_\emptyset+B_\emptyset)\Big)\cup (A_\emptyset+B_\emptyset)$ is an $H$-quasi-periodic decomposition with $\phi_H(A)$ an arithmetic progression having $\phi_H(A_0)$ as an end-term, it is now clear that $(hA\setminus hA_0)\cup hA_0$ is an $H$-quasi-periodic decomposition for any $h\geq k$. In particular, $(nA\setminus nA_0)\cup nA_0$ is an $H$-quasi-periodic decomposition, meaning (d) holds. Moreover, since $nA$ is aperiodic with $H$ nontrivial, we must have $nA_0$ aperiodic, so that (b) holds, as well as $|n\phi_H(A)|=n|\phi_H(A)|-n+1\leq |G/H|$. Consequently,  $$|nA|=|H|(|n\phi_H(A)|-1)+|nA_0|=|H|(n|\phi_H(A)|-n)+|nA_0|=n(|A|-|A_0|
+\epsilon)
+|nA_0|,$$ implying $|nA|-|A|n=|nA_0|-|A_0|n+\epsilon n$. Thus (e) holds. Since $nA_0$ is aperiodic, we have $|nA_0|<|\la A_0\ra_*|$ or $|A_0|=1$.
Finally, since $n(|A|-|A_0|
+\epsilon)
+|nA_0|=|nA|<(|A|+1)n-3$, it follows that $|nA_0|<(|A_0|+1-\epsilon)n-3$, and (c) holds, showing that (v) holds, which  completes the proof.
\end{proof}

For large $n$, most of the possibilities given by Theorem \ref{thm-main-sumset} are not possible, leading to the following \emph{non-recursive} description, which we will make use of (in the more specialized version stated in Corollary \ref{cor2}) for our generalization of Olson's result.

\begin{corollary}\label{cor1}
Let $G$ be a finite abelian group,  let $A\subseteq G$ be a nonempty subset with $\la A\ra_*=G$, let $n\geq 3$ be an integer, and let $K=\mathsf H(nA)$. If $n\geq \exp(G)+3$, then $$|nA|\geq \min\{|G|,\,(|A|+1)n-3\}.$$ If $n\geq \exp(G)-1$ and $|nA|< \min\{|G|,\,(|A|+1)n-3\}$, then one of the following holds.
\begin{itemize}
%\item[1.] If $n\geq \exp(G)+3$, then $|nA|\geq \min\{|G|,\,(|A|+1)n-3\}$.
\item[1.] $n=\exp(G)+2=7$, \ $G\cong C_5^2$, \ $|K|=1$, \  $|A|=3$ \  and  \ $|G|-1=|nA|=(|A|+1)n-4$ with $A$ given by Lemma \ref{lemma-itt-kst-size3}(ii)(c).
\item[2.]    $n=\exp(G)+1$, \  $4\mid \exp(G)$ and either
     \begin{itemize}
\item[(a)] $n=5$, \ $G\cong K\times C_4^2$, \ $|A|n\leq  \frac{15}{16}|G|$ \  and  \ $|G|-|K|=|nA|\geq |A+K|n$ with $\phi_K(A)$ given by Lemma \ref{lemma-itt-kst-size3}(ii)(b), or
\item[(b)] $n=\frac14|G|+1\geq 9$, \ $G\cong C_4\times C_{\exp(G)}$,  \  $|K|=1$, \ $|A|n= 3n= \frac34|G|+3<|G|$, \ and $|G|-1=|nA|=(|A|+1)n-5$ with $A$ given by Lemma \ref{lemma-itt-kst-size3}(ii)(b).
    \end{itemize}
 \item[3.]     $n=\exp(G)$, \ $G\cong H\times C_{\exp(G)}$ with $K<H$, \ $|A|n\leq |G|$, \ $|G|-|K|=|nA|\geq |A+K|n-|K|$, \ $|\phi_H(A)|=2$ and either
     \begin{itemize}
     \item[(a)] $H/K\cong C_2^2$ and  $(\phi_K(A),\phi_K(A))$ is elementary of type (VIII) with $|\phi_K(A)|=4$, or
     \item[(b)]  $|H/K|\geq 3$ and $z+A+K=(H\setminus K)\cup (x+K)$ for some $z\in G$ and $x\in G\setminus H$.
     \end{itemize}
\item[4.]     $n=\exp(G)-1$, \  $G\cong H\times C_{\exp(G)}$ with $K<H$ proper, $|\phi_H(A)|=2$, and either
\begin{itemize}
\item[(a)]  $|A|n\leq \frac{\exp(G)-1}{\exp(G)}|G|$, \ $|G|-|H|=|nA|\geq |A+K|n$, and 3(a) or 3(b) holds,
\item[(b)] $z+A+K=(H\setminus K)\cup (A_0+K)$ for some $z\in G$ with $A_0$ a nonempty subset of an $H$-coset, \ $|A|n\leq |G|-2|K|$, and $|G|-|H|+|n(A_0+K)|=|nA|\geq |A+K|n+|K|$,
\item[(c)] $z+A+K=H\cup (A_0+K)$ for some $z\in G$ with $A_0$ a nonempty subset of an $H$-coset, \  $|A|n\leq |G|$, and  $|nA|=|G|-|H|+|n(A_0+K)|$, or
\item[(d)] $G=H_0\oplus H_1\oplus\ldots \oplus H_r$ with  $K<H_0$ proper, $r\geq 1$ and $H_i=\la x_i\ra\cong C_{\exp(G)}$ for all $i\in [1,r]$, \ $z+A+K= \bigcup_{j=0}^{r}\big(K+\Sum{i=0}{j-1}H_i+\Sum{i=j+1}{r}x_i\big)$ for some $z\in G$, $|A|n\leq |G|-|H_0|+(\exp(G)-1)|K|\leq \frac{p\exp(G)^r+\exp(G)-p-1}{p\exp(G)^r}|G|$, where $p$ is the smallest prime divisor of $\exp(H_0)$, and $|nA|=|G|-|H_0|+|K|$.
\end{itemize}
\end{itemize}
\end{corollary}

\begin{proof} We may assume $$n\geq \max\{3,\,\exp(G)-1\},$$ as the corollary only applies in these cases.
Let $K=\mathsf H(nA)$, let $X=\phi_K(A)$ and suppose $|nA|<\min\{|G|,\,(|A|+1)n-3\}.$  If $|X|=|\phi_K(A)|=1$, then $nA=\la A\ra_*=G=K$, contrary to assumption. Therefore we can assume $|X|=|\phi_K(A)|\geq 2$.  In particular, $G/K$ is nontrivial.
Observe that $|nA|=|n(A+K)|=|nX||K|$. Thus, if $|nX|\geq x|X|+y$ for some integers $x\geq 0$ and $y$, then $|nA|\geq x|A|+y|K|$ as well. In particular, we have $|nX|<\min\{|G/K|,\,(|X|+1)n-3\}$ and can apply Theorem \ref{thm-main-sumset} to $nX$. We proceed to go through the possibilities for $X$ given by Theorem \ref{thm-main-sumset} one by one.

\subsection*{Case A} Suppose there is an arithmetic progression $P\subseteq G/K$ with $X\subseteq P$. Then $\la P\ra_*=\la X\ra_*=G/K$ is cyclic with $|G/K|\leq \exp(G)$. If $|P|=|X|$, then,
since $nX\neq G/K$ and $|X|\geq 2$, it follows that $n\leq |G/K|-2\leq \exp(G)-2$. If $|P|=|X|+1$, then $|X|\geq 3$ and $|nX|\geq |X|n-1\geq 3n-1$ (by the same calculations done out at the beginning of the proof of Theorem \ref{thm-main-sumset}), forcing $n\leq \frac13|G/K|\leq |G/K|-2\leq \exp(G)-2$. If $|X|=3$, $|P|=5$ and either $|nX|=4n-5=|G/K|-1$ or $|nX|=4n-4=|G/K|-1$, then $n\leq \frac14|G/K|+1\leq |G/K|-2\leq \exp(G)-2$. If $|P|=6$, $|G/K|=21$ and $|nX|=4n-4=|G/K|-1=20$, then $n=6<19=|G/K|-2\leq \exp(G)-2$. In all cases, we obtain the contradiction $n\leq \exp(G)-2$, thus handling all possibilities when  $X$ is contained in a short length arithmetic progression. In particular, the theorem is now established for $G\cong C_p$ with $p$ prime, allowing us to proceed by induction on $|G|$.

\subsection*{Case B} Suppose Theorem \ref{thm-main-sumset}(ii) holds for $X$, say \be\label{ref}z+A+K=\Big(x+K_1\Big)\cup \Big(y+H\Big)\cup \ldots\cup \Big((r-1)y+H\Big)\cup \Big(ry+K_2\Big),\ee where $H/K=K_1/K\oplus K_2/K\cong C_2\oplus C_2$ and $r\geq 1$.  Then  $G/H$ is cyclic since $\phi_H(A)$ is an arithmetic progression with $\la \phi_H(y)\ra=\la \phi_H(A)\ra_*=G/H$. Thus $|G/H|\leq \exp(G)$. Indeed, $|G/H|$ divides $\exp(G)$, in which case either $|G/H|=\exp(G)$ or $|G/H|\leq \frac12 \exp(G)$.
Since $nX\neq G/K$, we must have $3\leq n\leq |G/H|\leq \exp(G)$. If the latter bound is strict, we obtain the contradiction $3\leq n\leq \frac12\exp(G)\leq \exp(G)-2$. As a result, we must have $|G/H|=\exp(G)$, whence $G\cong H\times C_{\exp(G)}$, in which case $\exp(G)$ must be even (as $H$ contains a subgroup isomorphic to $C_2^2$). We also have either  $|nX|=|X|n\leq |G/K|-1$ or else $|nX|=|X|n-1=|G/K|-1$ (by Theorem \ref{thm-main-sumset}(ii)). Thus
$n\leq \frac{|G/K|}{|X|}\leq \frac14|G/K|=|G/H|=\exp(G)$. If $|X|>4$, then $|X|\geq 8$ (in view \eqref{ref}) and we can improve the bound to $3\leq n\leq \frac12\exp(G)\leq \exp(G)-2$, contrary to assumption. Therefore $|X|=4$ and $r+1=|\phi_H(A)|=2$.
%Thus $nA=G$ for $n>\exp(G)$, meaning we must have $n\leq \exp(G)$.
If
$n=\exp(G)-1$, then
 $|A|n\leq |X||K|n=4|K|(\exp(G)-1)=|G|-|H|=\frac{\exp(G)-1}{\exp(G)}|G|$ and $|nA|=|n(A+K)|=|nX||K|=n|X||K|=|A+K|n$ as well as $|nA|=|nX||K|=(|G/K|-|H/K|)|K|=|G|-|H|$, whence  4(a) holds. If $n=\exp(G)$, then $|A|n\leq |X||K|n=4n|K|=4\exp(G)|K|=|G|$ and $|nA|=|n(A+K)|=|nX||K|=(n|X|-1)|K|=|A+K|n-|K|$ as well as $|nA|=|nX||K|=(|G/K|-1)|K|=|G|-|K|$, whence 3(a) holds.

\subsection*{Case C} Suppose Theorem \ref{thm-main-sumset}(iii) holds for $X$, say
$$z+A+K=\{x+K\}\cup \Big(y+H\Big)\cup \ldots\cup \Big(ry+H\Big)\cup \Big((r+1)y+K\Big),$$ where $|H/K|=2$ and $r\geq 1$. Then $G/H$ is cyclic and generated by $\phi_H(y)$,  so $|G/H|\leq \exp(G)$. We have $|\phi_{H/K}(X)|=r+2\geq 3$. Thus, since $nX\neq G/K$ is aperiodic, we must have  $2n+1\leq|\phi_{H/K}(X)|n-n+1\leq |G/H|+1$ (the upper bound follows lest $nX=G/K$ in view of the structural description of $X$), implying $3\leq n\leq\frac12|G/H|\leq \frac12 \exp(G)\leq \exp(G)-2$, contrary to assumption.

\subsection*{Case D}Suppose Theorem \ref{thm-main-sumset}(iv) holds for $X$, say $$z+A=K\cup \Big(y+\big(H\setminus (x+K)\big)\Big)\cup \Big(2y+H\Big)\cup \ldots\cup \Big(ry+H\Big),$$ where $H/K$ is nontrivial and $r\geq 1$ with this inequality strict when $|H/K|=2$. Then $G/H$ is cyclic, generated by $\phi_H(y)$. We purposefully allow $|X|=3$ in this case, which corresponds to when $|H/K|=3$ and $r=1$. As in previous cases, $|G/H|$ divides $\exp(G)$, and thus $|G/H|<\exp(G)$ implies that $|G/H|\leq \frac12\exp(G)$. We have $|\phi_{H/K}(X)|=r+1\geq 2$ with the inequality strict when $|H/K|=2$.   Thus, since $nX\neq G/K$ is aperiodic, we  have  $rn+1\leq|\phi_{H/K}(X)|n-n+1\leq |G/H|+1$ (the upper bound follows lest $nX=G/K$ in view of the structural description of $X$).
 Consequently, if $r\geq 2$ or $|G/H|<\exp(G)$, we obtain the contradiction $3\leq n\leq \frac12 \exp(G)\leq \exp(G)-2$.
 Therefore we must have  $|G/H|=\exp(G)$ and $r=1$, in which case $|H/K|\geq 3$. Furthermore, $|\phi_H(A)|=2$, \ $G\cong H\times C_{\exp(G)}$ and $n\leq \exp(G)$. If $n=\exp(G)-1$, then $|G/K|-|H/K|=|nX|=|X|n$, \ $|G|-|H|=|nA|= |nX||K|=|A+K|n$ and $|A|n\leq n|X||K|=(\exp(G)-1)|H|=|G|-|H|=\frac{\exp(G)-1}{\exp(G)}|G|$, whence 4(a) holds. If $n=\exp(G)$, then $|G/K|-1=|nX|=|X|n-1$, \ $|G|-|K|=|nA|=|A+K|n-|K|$, and $|A|n\leq |X||K|n=|H|\exp(G)=|G|$. Thus 3(b) holds.

\subsection*{Case E} Suppose Theorem \ref{thm-main-sumset}(v) holds for $X$, say (after translating $A$ appropriately) $$A_0+K\subseteq A+K\subseteq P=(A_0+K)\cup (y+H)\cup \ldots\cup (ty+H),$$  with $H/K$ nontrivial, $A_0= H\cap A$ nonempty and $t\geq 1$.  Then  $G/H$ is cyclic,  generated by $\phi_H(y)$,  so $|G/H|\leq \exp(G)$. As in the previous cases, $|G/H|$ divides $\exp(G)$, and thus $|G/H|<\exp(G)$ implies that $|G/H|\leq \frac12\exp(G)$. Let $$X_0=\phi_K(A_0).$$
 In  view of Theorem \ref{thm-main-sumset}(vi)(a), we have $|\phi_K(P)|=|X|+\epsilon$ with $\epsilon \in \{0,1\}$. In  view of Theorem \ref{thm-main-sumset}(vi)(b), we have $n\phi_K(A_0)=nX_0$ aperiodic. In  view of Theorem \ref{thm-main-sumset}(vi)(d), we have $nA=nP$. We purposely allow $|X|=3$ in this case, which corresponds to when   $\epsilon =0$, $r=1$ and $|H/K|=2$. Note, if $\epsilon =1$, $|X_0|=1$,  $r=1$ and $|H/K|=3$, then we actually fall under Case D.

Since $nP=nA\neq G$, the structural description above ensures  $tn+1=n|\phi_H(A)|-n+1\leq |G/H|\leq \exp(G)$. Consequently, if $|G/H|\leq \frac12 \exp(G)$ or $t\geq 2$, then we obtain the contradiction $3\leq n\leq \frac{\exp(G)-1}{2}\leq \exp(G)-2$. Therefore we must have $t=1$ and $|G/H|=\exp(G)$, in which case $|\phi_H(A)|=2$, \ $G\cong H\times C_{\exp(G)}$ and $n+1\leq \exp(G)$, in turn implying $n=\exp(G)-1$.
%Since $nX_0$ is aperiodic, Kneser's Theorem implies that \be\label{X0-bound}n|X_0|-n+1\leq |nX_0|\leq |H/K|-1.\ee
Since $nA=nP$ with $n=\exp(G)-1$ and $r=1$, we have
$nA=\big(G\setminus (H+nA_0)\big)\cup nA_0$. In particular,
$\barr{nA_0}{H}=\barr{nA}{G}$,  \be\label{teecot}|nA|=|G|-|H|+|n(A_0+K)|\quad\und\quad\mathsf H(nA_0)=\mathsf H(\barr{nA_0}{H})=\mathsf H(\barr{nA}{G})=\mathsf H(nA)=K.\ee
In particular, $\mathsf H(nA_0)=K$ implies that  $nA_0=n(A_0+K)$.

If $\epsilon=1$, then
$z+A+K=(H\setminus K)\cup (z+A_0+K)$ for some $z\in -y+H$.
Since  $nA=nP$ and $K=\mathsf H(nA)=\mathsf H(nP)$, Kneser's Theorem implies \be\label{gogrimp}|nA|=|nP|\geq n|P+K|-(n-1)|K|=n|A+K|+|K|.\ee
Since $\mathsf H(nA_0)=K<H$ by \eqref{teecot}, we have $|nA_0|=|n(A_0+K)|\leq |H|-|K|$, which combined with \eqref{gogrimp} and \eqref{teecot} implies $|A|n\leq |A+K|n\leq |nA|-|K| \leq |G|-2|K|$. Thus 4(b) holds.

If $\epsilon=0$, then, by w.l.o.g. replacing $A$ and $A_0$ with appropriate translates, we have
$A+K=H\cup (A_0+K)$.  Letting $H'=\la A_0\ra_*\leq H$, we have $K\leq H'\leq H$ in view of \eqref{teecot}. If $|X_0|=1$, then $|nA_0|=|n(A_0+K)|=|K|$, \ $|A|\leq |X||K|=|H|+|K|$ and $|K|\leq \frac{1}{p}|H|$, where $p$ is the smallest prime divisor of $\exp(H)$ (since $K<H$ is a proper subgroup in view of $H/K$ being nontrivial). Thus
\begin{align*} |A|n\leq& \;(|H|+|K|)(\exp(G)-1)=|G|-|H|+(\exp(G)-1)|K|\\\leq & \; \frac{p+1}{p}(|G|-|H|)=\frac{p\exp(G)+\exp(G)-p-1}{p\exp(G)}|G|,\end{align*} in which case 4(d) holds with $H_0=H$ and $r=1$.  Therefore we may now assume $|X_0|>1$, whence $K<H'$ is a proper subgroup and $H'/K$ is nontrivial. In particular, since $nX_0$ is aperiodic, we must have $|nX_0|<|H'/K|\leq |H/K|$. Thus, if $|nX_0|\geq |X_0|n$, then we have $|X_0|n\leq |nX_0|\leq  |H/K|-1$, whence $|A|n\leq |H|n+|X_0||K|n\leq |G|-|K|$, meaning 4(c) holds. Therefore we may instead assume \be\label{inductallow}|nX_0|<\min\{|H'/K|,\,|X_0|n\}\quad\und\quad|nA_0|<
\min\{|H'|,\,|A_0+K|n\},\ee where the second inequality follows by multiplying the first by $|K|$.
Since $3\leq n=\exp(G)-1$, we have $\exp(G)\geq 4$.
 If $\exp(H'/K)<\exp(G)$, then $\exp(H'/K)\leq \frac12\exp(G)$ and $n=\exp(G)-1\geq \frac12\exp(G)+1\geq \exp(H'/K)+1$. Consequently, applying the induction hypothesis to $nX_0$ shows either Item 1 or 2 holds for $X_0$, in which case $|nX_0|\geq |X_0|n$, contrary to \eqref{inductallow}. Therefore we instead conclude that $\exp(H'/K)=\exp(H')=\exp(H)=\exp(G)$. But now \eqref{inductallow} and \eqref{teecot}  combined with an application of the induction hypothesis to $n(A_0-x)$, where $x\in A_0$ is any element, imply either $|A_0|n\leq |H'|\leq |H|$ or else that 4(d) holds for $n(A_0-x)$. In the former case, we have $|A|n\leq |H|n+|A_0|n=|G|-|H|+|A_0|n\leq |G|$, whence 4(c) holds. On the other hand, in the latter case, we have $H'=H_0\oplus H_1\oplus\ldots \oplus H_{r-1}$ with $K<H_0$ proper, $r-1\geq 1$, $H_i=\la x_i\ra\cong C_{\exp(H')}\cong C_{\exp(G)}$ for all $i\in [1,r-1]$, \ \be\label{waith}z+A_0-x+K=\bigcup_{j=0}^{r-1}
 \big(K+\Sum{i=0}{j-1}H_i+\Sum{i=j+1}{r-1}x_i\big)\ee for some $z\in H'$, \be\label{bounddou}|A_0|n\leq |H'|-|H_0|+(\exp(G)-1)|K|\leq \frac{p\exp(G)^{r-1}+\exp(G)-p-1}{p\exp(G)^{r-1}}|H'|,\ee where $p$ is the smallest prime divisor of $\exp(H_0)$,
 and \be\label{eagle}|nA_0|=|H'|-|H_0|+|K|.\ee If $H'\neq H$, then $|H'|\leq \frac12|H|$, in which case \eqref{bounddou} implies $|A_0|n<|H|$ and $|A|\leq |H|n+|A_0|n<|H|(n+1)=|G|$, whence 4(c) holds. Therefore we may assume $H'=H$. Since $x\in A_0$ and since $\la \phi_H(A_0)\ra=\la \phi_H(A)\ra=G/H$ (recall $A+K=H\cup (A_0+K)$) with $|G/H|=\exp(G)$, we have $G=H\oplus \la x\ra$. Thus, letting $x_r=x$ and $H_r=\la x_r\ra=\la x\ra\cong C_{\exp(G)}$, we find that $G=H\oplus H_r=H'\oplus H_r=H_0\oplus H_1\oplus\ldots\oplus H_r$. In view of \eqref{waith}, $z\in H'=H$ and $K<H_0$, we have $z+A+K=(z+H)\cup (z+A_0+K)=H\cup (z+A_0+K)=\bigcup_{j=0}^{r}\big(K+\Sum{i=0}{j-1}H_i+\Sum{i=j+1}{r}x_i\big)$.
 In view of \eqref{bounddou} and $H'=H$, we have $|A|n\leq |H|n+|A_0|n=|G|-|H|+|A_0|n\leq |G|-|H_0|+(\exp(G)-1)|K|\leq |G|-|H_0|+(\exp(G)-1)\frac{|H_0|}{p}= \frac{p\exp(G)^r+\exp(G)-p-1}{p\exp(G)^r}|G|$, where $p$ is the smallest prime divisor of $\exp(H_0)$ (since $K<H_0$ is proper).  In view of \eqref{eagle}, \eqref{teecot} and $H'=H$, we have $|nA|=|G|-|H|+|nA_0|=|G|-|H_0|+|K|$. Thus 4(d) holds.

\subsection*{Case F} Suppose that $|X|=3$ and Lemma \ref{lemma-itt-kst-size3}(ii)  holds. Then  there is an $H/K$-coset decomposition $X=X_1\cup X_0$ with $|X_1|=2$, $|X_0|=1$ and  $\la X_1\ra_*=H/K$ cyclic. Then $G/H$ is generated by a non-zero difference from $\phi_H(X_1-X_0)$, ensuring that $G/H$ is cyclic, whence $|G/H|\leq \exp(G)$ with $|G/H|\leq \frac12\exp(G)$ when equality fails (as in previous cases).

If $2\leq |H/K|\leq 3$, then Theorem \ref{thm-main-sumset}(iv) or (v) holds, which was handled in Cases D and E. If $|H/K|=4$ and $|nX|=4n-5=|G/K|-1$, then $4\mid \exp(G)$ (as $H/K$ is cyclic of order $4$), \ $$|nA|=|G|-|K|$$  and $3\leq n=\frac14|G/K|+1=|G/H|+1\leq \exp(G)+1$. If the latter inequality is strict, we obtain the contradiction $3\leq n=|G/H|+1\leq \frac12\exp(G)+1\leq \exp(G)-2$ unless $\exp(G)=4$, $|G/H|=\frac12\exp(G)=2$ and $n=3=\frac14|G/K|+1$. In this case, since $nX\neq G/K$, Theorem \ref{PigeonHoleBound} implies that $|2X|\leq |G/K|-3=5$. Translating as necessary, we can w.l.o.g. assume $X_1=\{0,x\}$ and $X_0=\{y\}$. Thus $2X=\{0,x,2x\}\cup \{y,y+x\}\cup \{2y\}$. Since $|2X|\leq 5$ and $|H/K|=4$, we must have $2y\in \{0,x,2x\}$. If $2y=x$, then $X=\{0,y,2y=x\}$ is an arithmetic progression, which was handled in Case A. If $2y=2x$, then $\ord(x-y)=2$ and $X=\{x,y\}\cup \{0\}$ is an $H'/K$-coset decomposition with $|H'/K|=2$. Thus Theorem \ref{thm-main-sumset}(v) holds, which was handled in Case E.  Finally, if $2y=0$, then $\ord(y)=2$ and $X=\{0,y\}\cup \{x\}$ is an $H'/K$-coset decomposition with $|H'/K|=2$, in which case Theorem \ref{thm-main-sumset}(v) again holds, which was handled in Case E.
Thus we can instead assume $|G/H|=\exp(G)$ and $n=|G/H|+1=\exp(G)+1\geq 5$. In particular, $G\cong H\times C_{\exp(G)}$. We also have  $$|A+K|n+(n-5)|K|= (4n-5)|K|=|nX||K|=|nA|\leq |A|n+n-4\leq 3n|K|+n-4,$$ which (in view of $n\geq 5$) implies either $|K|=1$ or $n\leq 6$. Since $n=\exp(G)+1$ with $4\mid \exp(G)$, we conclude that $n=6$ is not possible.
If $n=5$, then $|G/H|=\exp(G)=4$ and $|A|n\leq 15|K|\leq \frac{15}{16}|G|$. Moreover, since $H/K$ is a cyclic group of order $4=\exp(G)$ with $G\cong H\times C_4$, it follows that  $\exp(H)=\exp(G)=4$ with $H\cong K\times C_4$. Thus 2(a) holds. On the other hand, if $n>5$, then  $|K|=1$, \ $|H|=4$, \ $\frac14|G|+1=|G/H|+1=n\geq 9$,   and $|A|n=3n=3\exp(G)+3= 3|G/H|+3=\frac34|G|+3<|G|$. Moreover, since $H/K=H$ is a cyclic group of order $4$, we have $G\cong H\times C_{\exp(G)}\cong C_4\times C_{\exp(G)}$. Thus 2(b) holds.

If $|H/K|=|G/H|=5$, $n=7$ and $|nX|=4n-4=|G/K|-1=24$, then $5\mid \exp(G)$. We must have $5=|G/H|=\exp(G)$, for otherwise $\exp(G)\geq 2|G/H|=10$, contradicting that $7=n\geq \exp(G)-1$. Hence $G\cong C_5^s$,  $H\cong C_5^{s-1}$ and $K\cong C_5^{s-2}$, where $s\geq 2$. We now have $(|A|+1)n-3>|nA|=|nX||K|=24|K|= 3n|K|+3|K|\geq n|A|+3|K|$, implying $3|K|<4$, and thus $|K|=1$. Hence $G\cong C_5^2$, \ $|A|=3$, and $|G|-1=|nA|=24=(|A|+1)n-4$. Thus 1 holds.

\subsection*{Case G} Suppose Lemma \ref{lemma-itt-kst-size3}(iii) holds for $X$, in which case  there is an $H/K$-coset decomposition $X=\{x,z\}\cup \{y\}$ with $2(y+z)=4x$ and  $\la x-z\ra=H/K\leq G/K$ a cyclic subgroup such that $|G/H|=2$. Moreover, $G/K\cong C_2\times C_m$ with $m$ even. Hence $\exp(G)\geq \exp(G/K)= \frac12|G/K|$. We have $|nX|=4n-5=|G/K|-1$, ensuring that $3\leq n=\frac14|G/K|+1=\frac12\exp(G/K)+1\leq\frac12\exp(G)+1$, which  contradicts that $n\geq \exp(G)-1$ unless $m=\exp(G/K)=\exp(G)=4$. Thus $G/K\cong C_2\times C_4$ and $n=\frac14|G/K|+1=3$. By lemma \ref{lemma-itt-kst-size3}(iii), we also have $2(y-z)=4(x-z)=0$, with the latter equality in view of $\exp(G/K)=4$. Thus $X=\{y,z\}\cup \{x\}$ is an $H'/K$-coset decomposition with $H'/K=\la y-z\ra$ a subgroup of order $2$, in which case Theorem \ref{thm-main-sumset}(v) holds with $\epsilon=0$ and $|X_0|=1$, which was handled in Case E.

\subsection*{Case H} Suppose  Lemma \ref{lemma-itt-kst-size3}(iv) holds for $X$, in which case $G/K$ is cyclic and  $4n-5=|nX|=|G/K|-1$. Thus $3\leq n=\frac14|G/K|+1\leq \frac14 \exp(G)+1\leq \exp(G)-2$, contrary to assumption.
As this exhausts the last possibility for $X$, the proof is now complete.
\end{proof}

When $|A|$ is large, the previous corollary simplifies drastically.

\begin{corollary}\label{cor2}
Let $G$ be a finite abelian group,  let $A\subseteq G$ be a nonempty subset with $\la A\ra_*=G$, let $n\geq 1$ be an integer,  let $K=\mathsf H(nA)$ and suppose $n|A|>|G|.$
\begin{itemize}
\item[1.] If $n\geq \exp(G)$, then $nA=G$.
\item[2.] If $n=\exp(G)-1$ and $nA\neq G$, then $\exp(G)$ is composite, $G=H_0\oplus H_1\oplus\ldots \oplus H_r$ with   $K<H_0$ proper, $r\geq 1$ and $H_i=\la x_i\ra\cong C_{\exp(G)}$ for all $i\in [1,r]$ (thus $G$ is non-cyclic),  $$z+A+K= \bigcup_{j=0}^{r}\big(K+\Sum{i=0}{j-1}H_i+\Sum{i=j+1}{r}x_i\big)\quad\mbox{ for some $z\in G$},$$ $|A|n\leq |G|-|H_0|+(\exp(G)-1)|K|\leq \frac{p\exp(G)^r+\exp(G)-p-1}{p\exp(G)^r}|G|$, where $p$ is the smallest prime divisor of $\exp(H_0)$, and $|nA|=|G|-|H_0|+|K|$.
\end{itemize}
\end{corollary}

\begin{proof}
Since $n|A|>|G|$, we must have $n\geq 2$. If $n=2$, then Theorem \ref{PigeonHoleBound} and $2|A|=n|A|>|G|$ implies $nA=G$, either as desired or contrary to hypothesis. Therefore we can assume $n\geq 3$. We may assume $nA\neq G$, as there is nothing to prove otherwise, in which case $|nA|<|G|<n|A|\leq n|A|+n-3$, allowing us to apply Corollary \ref{cor1}. We observe that $|nA|\geq |A|n\geq |G|$ for all possibilities with $n\geq \exp(G)+1$. If $n=\exp(G)$, all possibilities from Corollary \ref{cor1} have $|A|n\leq |G|$, contrary to hypothesis. This establishes Item 1. Next suppose that $n=\exp(G)-1$. Then the hypothesis $|A|n>|G|$ means that Corollary \ref{cor1}.4(d) must hold with $\exp(G)$ composite, else $|A|n\leq\frac{p\exp(G)^r+\exp(G)-p-1}{p\exp(G)^r}|G|=
\frac{p\exp(G)^r-1}{p\exp(G)^r}|G|<|G|$, and now Item 2 follows.
\end{proof}

\section{Subsequence Sums}\label{sec-subsums}

In this section, we provide the proofs for Theorems \ref{thm-main-olson} and \ref{thm-main-nsums}. We begin with a lemma that can be combined with the Partition Theorem to show that only one of two extremes is possible for the subgroup $\la X\ra_*$ (where $X$ is as defined in Theorem \ref{thm-partition-thm}).

\begin{lemma}\label{lem-partition-extra}
Let $G$ be an abelian group, let $n\geq 1$, let $S\in \Fc(G)$ be a sequence, let $S'\mid S$ be a subsequence  with $\mathsf h(S')\leq n\leq |S'|$, let $H=\mathsf H(\Sigma_n(S))$,   let $X\subseteq G/H$ be the subset of all $x\in G/H$ for which $x$ has multiplicity at least $n$ in $\phi_H(S)$, and let $Z=\phi_H^{-1}(X)$. Suppose $|\Sigma_n(S)|<|S'|-n+1$. Then either $$\la Z\ra_*=H\quad\mbox{ or }\quad\la Z\ra_*=\la \supp(S)\ra_*.$$
\end{lemma}

\begin{proof}
By translating the terms of $S$ appropriately, we can w.l.o.g. assume $0\in \supp(S)\cap Z$. Note $X\neq \emptyset$, and thus also $Z\neq \emptyset$, in view of $|\Sigma_n(S)|<|S'|-n+1$ (as remarked immediately after Theorem \ref{thm-partition-thm}).  Let $L=\la Z\ra_*=\la Z\ra$.  Since $|\Sigma_n(S)|<|S'|-n+1$, we can apply Theorem \ref{thm-partition-thm}.2 to $\Sigma_n(S)$  and let $\sA=A_1\bdot\ldots\bdot A_n$ be the resulting setpartition. Then $\Sigma_n(S)=\Sum{i=1}{n}A_i$, \ $H$ is nontrivial, $Z\neq \emptyset$, and \be\label{bound}|S'|-n\geq |\Sum{i=1}{n}A_i|\geq |S'|-n+1-(n-e-1)(|H|-1)+\rho,\ee with $e$ and $\rho\geq 0$ as defined in Theorem \ref{thm-partition-thm}. Note \eqref{bound} implies that $e\leq n-2$. Since $H=\mathsf H(\Sigma_n(S))$, we have $H\leq \la \supp(S)\ra_*$, while $H\leq \la Z\ra_*=L$ follows by definition of $Z$. Thus to prove $L=\la Z\ra_*=\la \supp(S)\ra_*$,  it suffices to show $\la \phi_H(\supp(S))\ra_*=\la \phi_H(Z)\ra_*=\la X\ra_*=L/H$. Since $\phi_H(Z)\subseteq \phi_H(\supp(S))$, the inclusion $L/H\leq
\la \phi_H(\supp(S))\ra_*$ is trivial. Assuming by contradiction that  the reverse inclusion is false, then  there must some $x\in \supp(S)\setminus L$. Re-index the $A_i$ so that $\phi_H(A_i)=X\subseteq L/H$ for $i=1,\ldots,k$ and $A_{k+1}\nsubseteq L$, where $k=n-e\geq 2$.

Let $N=|X|$. We may assume $L/H=\la X\ra_*$ is nontrivial and $N\geq 2$, else $H=L$ follows, yielding the other desired conclusion. Note $$|\phi_H(A_i)|=N\;\mbox{for $i\in [1,k]$}\quad\und\quad |\phi_H(A_i)|=N+1\;\mbox{ for $i\in [k+1,n]$}$$ in view of Theorem \ref{thm-partition-thm}.2 and our choice of indexing.
Since $H=\mathsf H(\Sigma_n(S))=\mathsf H(\Sum{i=1}{n}A_i)$, it follows that $\Sum{i=1}{n}\phi_H(A_i)$ is aperiodic. In particular, $kX$ is aperiodic, whence Kneser's Theorem implies that $|\Sum{i=1}{k}\phi_H(A_i)|=|kX|\geq kN-k+1$. Since $0\in X\subseteq \phi_H(A_{k+1})\nsubseteq L/H$ and $\Sum{i=1}{k}A_i\subseteq L$, we have $|\Sum{i=1}{k+1}\phi_H(A_i)|\geq 2|\Sum{i=1}{k}\phi_H(A_i)|\geq 2kN-2(k-1)$. Since $\Sum{i=1}{n}\phi_H(A_i)$ is aperiodic, Kneser's Theorem implies \begin{align*}\nn|\Sum{i=1}{n}\phi_H(A_i)|\geq&\; |\Sum{i=1}{k+1}\phi_H(A_i)|+\Sum{i=k+2}{n}|\phi_H(A_i)|-(n-k)+1\\\nn \geq& \; 2kN-2(k-1)+(n-k-1)(N+1)-n+k+1\\\nn =&\;
(n+k-1)N-2(k-1).
\end{align*}
Thus, since $H=\mathsf H(\Sum{i=1}{n}A_i)$, it follows that $|\Sum{i=1}{n}A_i|\geq (n+k-1)N|H|-2(k-1)|H|$. By hypothesis, $|\Sum{i=1}{n}A_i|=|\Sigma_n(S)|\leq |S'|-n\leq nN|H|+e-n=nN|H|-k$. Combining this with the previous estimate, we obtain $nN|H|-k\geq (n+k-1)N|H|-2(k-1)|H|$, implying $-k\geq (k-1)N|H|-2(k-1)|H|\geq 0$, where the final inequality makes use of $N\geq 2$. But since $k=n-e\geq 2$, this is a contradiction, completing the proof.
\end{proof}

\begin{proof}[Proof of Theorem \ref{thm-main-nsums}]
Let $H=\mathsf H(\Sigma_n(S))$, let $X\subseteq G/H$ be the subset of all $x\in G/H$ for which $x$ has multiplicity at least $n$ in $\phi_H(S)$, and let $Z=\phi_H^{-1}(X)\subseteq G$.
Apply Theorem \ref{thm-partition-thm} to $\Sigma_n(S)$ using $S'=S$ and let $\sA=A_1\bdot\ldots\bdot A_n$ be the resulting setpartition. If $|\Sum{i=1}{n}A_i|\geq \Sum{i=1}{n}|A_i|-n+1=|S|-n+1\geq |G|$, then $\Sigma_n(S)=G$ follows from Theorem \ref{thm-partition-thm}. Therefore we can assume Theorem \ref{thm-partition-thm}.2 holds.
Thus, letting $N=|X|$ and $e=\Sum{i=1}{n}|A_i\setminus Z|$, it follows that \be\label{gogag} (|S|-n+1)-(n-e-1)(|H|-1)\leq ((N-1)n+e+1)|H|\leq |\Sigma_n(S)|=|\Sum{i=1}{n}A_i|\leq |G|-|H|,\ee else the desired conclusion $\Sigma_n(S)=G$ follows. In particular, $e\leq n-2$ in view of $|S|-n+1\geq |G|$, and $H<G$ is a nontrivial subgroup. We also must have $N\geq 1$, else $e=|S|$ follows, in which case  \eqref{gogag} implies  $|\Sigma_n(S)|\geq (|S|-n+1)|H|\geq |G|$, contrary to assumption.
Thus $X$ is nonempty. If $N=1$, then \eqref{gogag} implies that $e\leq |G/H|-2$, contrary to the coset condition hypothesis. Therefore we must have $N=|X|\geq 2$. By translating, we can w.l.o.g. assume $0\in Z\cap \supp(S)$. By re-indexing the $A_i$, we can also assume $$\phi_H(A_i)=X\quad\mbox{ for $i=1,\ldots,k$},$$ where $k=n-e\geq 2$.

We must have $\la \supp(S)\ra_*=G$, for if $L=\la \supp(S)\ra_*<G$ is a proper subgroup, then all but $0$ terms of $S$ are from the subgroup $L$ with $0\leq |G/L|-2$, contrary to hypothesis. Consequently, if $\la Z\ra_*<G=\la \supp(S)\ra_*$ is proper, then, since $|X|=|\phi_H(Z)|\geq 2$, Lemma \ref{lem-partition-extra} implies that $|\Sigma_n(S)|\geq |S|-n+1\geq |G|$, contrary to hypothesis. Therefore we  instead conclude that  \be \label{gob1}\la Z\ra_*=G\quad\und\quad\la X\ra_*=G/H.\ee

Assume by contradiction that  $\Sigma_n(S)\neq G$. Then, in view of  $H=\mathsf H(\Sum{i=1}{n}A_i)$, we have  \be\label{nXnotGH} nX\neq G/H,\ee for otherwise $G/H= nX\subseteq \Sum{i=1}{n}\phi_H(A_i)$, implying $G=\Sum{i=1}{n}A_i=\Sigma_n(S)$, contrary to assumption.  Since $e\leq n-2$, we have $|Z|n\geq |S|-e\geq (n+|G|-1)-(n-2)>|G|$. Thus, since $|X|\geq |Z|/|H|$, we conclude that \be\label{gob2}|X|n>|G/H|.\ee
%If $|nX|\geq |X|n$, then $|\Sum{i=1}{n}A_i|=|\Sum{i=1}{n}\phi_H(A_i)||H|\geq |X||H|n\geq |S|-e\geq |S|-n+2\geq |G|$, contrary to assumption. Therefore we must have \be\label{gob3}|nX|<\min\{|G/H|,\,|X|n\}.\ee

If $n=1$, then $|S|\geq |G|+n-1$ and $\mathsf h(S)\leq n$ imply $\supp(S)=G$, whence $\Sigma_n(S)=G$ follows, contrary to assumption. If $n=2$, then $|S|\geq |G|+n-1\geq |G|+1$ and $\mathsf h(S)\leq 2$. Applying Theorem \ref{PigeonHoleBound} to $A_1+A_2$ yields $|\Sigma_n(S)|=|\Sum{i=1}{n}A_i|=|A_1+A_2|=|G|$, contrary to assumption. Therefore we must have $n\geq 3$.

If $n\geq \exp(G)\geq \exp(G/H)$, we can apply Corollary \ref{cor2}.1 to $nX$ (in view of \eqref{gob1} and \eqref{gob2}) to obtain $nX=G/H$, contradicting \eqref{nXnotGH}. Thus Item 1 is complete.

If $n\geq \exp(G)-1\geq \exp(G/H)-1$, we can apply Corollary \ref{cor2} to $nX$ to conclude $\exp(G/H)=\exp(G)$ is composite and $G/H$ is non-cyclic. This completes Item 2 when $\exp(G)$ is prime. Moreover, if $G\cong H'\oplus C_{\exp(G)}$ with $|H'|$ prime, then $\exp(G/H)=\exp(G)$ is only possible if $G\cong H\oplus  C_{\exp(G)}$ with $G/H\cong C_{\exp(G)}$  cyclic, contrary to assumption. Thus Item 2 is complete in all cases.

If $G$ is cyclic and $n\geq \frac1p|G|-1$, where $p$ is the smallest prime divisor of $|G|$, then $n\geq |G/H|-1=\exp(G/H)-1$ follows in view of $H$ being nontrivial. Then, since $G/H$ is cyclic, Corollary \ref{cor2} implies that $nX=G/H$, contrary to hypothesis.

If $\exp(G)\leq 3$, then $n\geq 3\geq \exp(G)$, in which case Item 1 implies that Item 4 holds. If $|G|<10$, then $|G/H|\leq 4$ (since $H$ is nontrivial). Thus \eqref{gogag} yields the contradiction $|G|>|\Sum{i=1}{n}A_i|\geq ((N-1)n+1)|H|\geq (n+1)|H|\geq 4|H|\geq |G|$.
\end{proof}

Next, we give the proof of Theorem \ref{thm-main-olson}, which follows rather quickly using Theorem \ref{thm-main-nsums}.

\begin{proof}[Proof of Theorem \ref{thm-main-olson}]
Let $n=|S|-|G|$. Then $\Sigma_{|G|}(S)=\sigma(S)-\Sigma_n(S)$ and $|S|=n+|G|$. Thus Theorem \ref{thm-main-olson} follows immediately from Theorem \ref{thm-main-nsums} except when $\exp(G)=4$ and $|G|=16$, or when $|G|=10$.  Assume by contradiction that $\Sigma_{|G|}(S)\neq G$, and thus $\Sigma_n(S)\neq G$ as well.
%If $n=1$, then $\mathsf h(S)\leq n$ and $|S|= |G|+n$ cannot both hold, meaning $n\geq 2$.
If $n=2$, then $\mathsf h(S)\leq n=2$ and $|S|\geq |G|+n=|G|+2$ imply $\Sigma_n(S)=G$ via Theorem \ref{PigeonHoleBound} applied to the sets from any set partition $\mathcal A=A_1\bdot A_2$ with $\mathsf S(\mathcal A)=S$. Therefore we must have $n\geq 3$. Assume by contradiction that $|\Sigma_n(S)|<|G|< |S|-n+1$.
We proceed as in the proof of Theorem \ref{thm-main-nsums}, including all notation used there, e.g., $H$, $X\subseteq G/H$, $Z$, $\A=A_1\bdot\ldots\bdot A_n$, $N$ and $e$. In particular, we again conclude that $H<G$ is proper and nontrivial,  $nX\neq G/H$, \ $e\leq n-2$, $|X|=N\geq 2$,  $\la X\ra_*=G/H$ and $|X|n>|G/H|$, allowing us to apply Corollary \ref{cor2} to $nX$. By Theorem \ref{thm-partition-thm} applied to $\Sigma_n(S)$ with $S=S'$, we have \be\label{grabto}  ((N-1)n+e+1)|H|\leq |\Sigma_n(S)|\leq |G|-|H|.\ee

If $|G|=10$, then $G$ is cyclic and $\exp(G/H)=|G/H|\in \{2,5\}$. If $\exp(G/H)=2<n$, then Corollary \ref{cor2}.1 implies $nX=G/H$, contrary to assumption. Therefore $|H|=2$ and $\exp(G/H)=5$ is prime, whence $G/H$ is cyclic. Thus, if $n\geq 4=\exp(G/H)-1$, then Corollary \ref{cor2} again gives the contradiction $nX=G/H$. Therefore we must have $n=3$. If $N\geq 3$, then Theorem \ref{PigeonHoleBound} implies that $2X=G/H$, contradicting that $nX=3X\neq G/H$. Therefore $N=|X|=2$, which combined with $e\leq n-2=1$ ensures that $13=|G|+n=|S|\leq n|H|N+e\leq 13$, forcing equality to hold in all estimates. In particular, $e=1$. But then \eqref{grabto} yields the contradiction $10=2(e+4)=(n+e+1)|H|\leq 8$.

If  $\exp(G)=4$ and $|G|=16$, then Corollary \ref{cor2} yields the contradiction $nX=G/H$ unless $n=3=\exp(G/H)-1$ with $G/H$ non-cyclic, which is only possible if $|H|=2$ and $G/H\cong C_2\times C_4$.
 In this case, Corollary \ref{cor2}.2 instead implies
 $3|X|=|X|n\leq \frac{2\exp(G/H)+\exp(G/H)-2-1}{2\exp(G/H)}|G/H|=9$. Hence $|X|\leq 3$, \ $e\leq n-2=1$  and $19=|G|+n=|S|\leq |X||H|n+e\leq 18+1=19$. We therefore conclude that equality must hold in all estimates, in which case $e=1$ and  $N=|X|=3$. But then \eqref{grabto} yields the contradiction $16=2(e+7)=(2n+e+1)|H|\leq 14$.
\end{proof}

\subsection*{Acknowledgements} I thank the referees for the  suggestions that helped improve the presentation of the paper and main results.

\end{document}